\documentclass[10pt]{amsart}
\usepackage{amsmath,amssymb,amsthm}

\usepackage{float}
\usepackage[all,cmtip]{xy}
\usepackage{tikz}
\usetikzlibrary{matrix}
\usetikzlibrary{cd}
\usepackage{subfig}
\usepackage[colorlinks,hyperindex,linkcolor=blue]{hyperref}
\hypersetup{
 pdfauthor = {Francisco Javier Martinez Aguinaga},
 pdftitle= {Existence and classification of maximal growth distributions},
 pdfsubject = {Differential Geometry, Geometric Topology, distributions, h-principle}}

\newcommand{\R}{{\mathbb{R}}}

\newcommand{\SD}{{\mathcal{D}}}

\newcommand{\Hom}{{\text{Hom}}}
\newcommand{\Sym}{{\text{Sym}}}

\newcommand{\SF}{{\mathcal{F}}}

\newcommand{\SL}{{\mathcal{L}}}

\newcommand{\NS}{{\mathbb{S}}}

\newcommand{\Gr}{{\operatorname{Gr}}}

\newcommand{\rank}{{{\operatorname{rank}}}}
\newcommand{\Op}{{\mathcal{O}p}}

\newcommand{\SR}{{\mathcal{R}}}
\newcommand{\step}{{{\operatorname{step}}}}

\renewcommand{\Pr}{{{\operatorname{Pr}}}}
\newcommand{\Conv}{{\operatorname{Conv}}}

\newcommand{\symb}{{\mathsf{symb}}}

\newcommand{\Diff}{{\operatorname{Diff}}}

\newcommand{\Sec}{\text{Sec}}

\usepackage[framemethod=TikZ]{mdframed}
\newcounter{theo}[section]

\newtheorem{proposition}{Proposition}
\newtheorem{theorem}[proposition]{Theorem}
\newtheorem{definition}[proposition]{Definition}
\newtheorem{lemma}[proposition]{Lemma}

\newtheorem{corollary}[proposition]{Corollary}
\newtheorem{remark}[proposition]{Remark}
\newtheorem{question}[proposition]{Question}
\newtheorem{example}[proposition]{Example}

\makeatletter
\newcommand{\superimpose}[2]{%
  {\ooalign{$#1\@firstoftwo#2$\cr\hfil$#1\@secondoftwo#2$\hfil\cr}}}
\makeatother

\title{Existence and classification of maximal growth distributions}

\author{Javier Mart\'inez-Aguinaga}

\address{Universidad Complutense de Madrid, Departamento de Algebra, Geometría y Topología, Facultad de
Matemáticas, and Instituto de Ciencias Matemáticas CSIC-UAM-UC3M-UCM, C. Nicolás Cabrera, 13-15,
28049 Madrid, Spain}
\email{xabima@hotmail.com}

\emergencystretch=\maxdimen
\begin{document}

\begin{abstract}

This article tackles the problem of existence and classification of maximal growth distributions on smooth manifolds. We show that maximal growth distributions of $\rank>2$ abide by a full $h$-principle in all dimensions. We make use of M. Gromov's higher order convex integration and, on the way, we establish a new criterion for checking ampleness of a differential relation. 

As a consequence we answer in the positive, for $k>2$, the long-standing open question posed by M. Kazarian and B. Shapiro more than 25 years ago in \cite{KS} of whether any parallelizable manifold admits a $k$-rank distribution of maximal growth. We also answer several related open questions.

For completeness we show that the differential relation of maximal growth for rank-$2$ distributions is not ample in any ambient dimension. Non-ampleness of the Engel and the $(2,3,5)$-conditions follow as particular cases.

\end{abstract}
\maketitle
\addtocontents{toc}{\setcounter{tocdepth}{1}}

\section{Introduction}\label{ssec:intro}

This article tackles the long-standing open problem of existence and classification of maximal growth distributions on smooth manifolds. This question has garnered much interest in the mathematical community in recent years.

Every generic distribution germ has maximal growth \cite{AN}; i.e. \emph{locally}, Lie brackets of vector fields tangent to generic distributions generate new directions as fast as possible. This condition implies, in particular, that such vector fields eventually generate the whole tangent bundle by Lie bracket operation.

This situation drastically changes when one passes to the \emph{global} setting. Indeed, there are global obstructions to the existence of distributions of maximal growth. The existence of orientable Engel distributions (maximal growth rank-$2$ distributions on $4$-manifolds) requires, for example, parallelizability of the ambient manifold \cite{KMS}. A breakthrough in this direction was the paper \cite{Vo} by T. Vogel, where he showed that parallelizability was also a sufficient condition. Later on, there has been further development regarding the classification problem of Engel structures (see \cite{CPPP, Pino, CPV, CPP, PV}).

Some other remarkable and well studied distributions are contact and even-contact structures. These are maximally non-integrable hyperplane fields, the former in odd-dimensional manifolds whereas the latter in even-dimensional manifolds. D. McDuff showed in \cite{McD} that even-contact structures abide by a full $h$-principle, providing thus a classification result from a homotopic viewpoint. This result follows from Gromov's convex integration and contrasts with the case of contact structures.

 In general, contact structures do not abide by a complete $h$-principle unless one restricts to particular subclasses called \em overtwisted\em.  M. S. Borman, Y. Eliashberg and E. Murphy showed in \cite{BEM} that there is a one-to-one correspondence between formal classes of contact structures and overtwisted classes. The existence of tight contact distributions (non-overtwisted) thus shows that an $h$-principle for contact structures cannot hold in general. See \cite{Ben} for their existence in dimension $3$. As for the higher dimensional case, see, for instance, K. Niederkrüger's work \cite{Niederkruger}, which implies tightness of Liouville fillable contact structures.

The existence and classification problem for maximal growth distributions on open manifolds follows from Gromov's \textit{$h$-principle for open manifolds} \cite[p. 79]{Gro86}. See also \cite[Thms. 7.2.3 and 7.2.4]{EM}. Nonetheless, not so much was known for the case of closed manifolds and several open questions have been posed in the literature throughout the years about this problem.  

 In \cite{KS}, M. Kazarian and B. Shapiro described some topological obstructions for maximal growth distributions to exist in general. Regarding these results, and very much in the spirit of the $h$-principle philosophy, they comment: 

\begin{quote} ``The basic problem related to the above topological obstructions is to what extent
vanishing of these obstructions guarantees the existence of a maximal growth subbundle. The result of T. Vogel on the existence of Engel structures on parallelizable $4$-manifolds brings a certain amount of optimism about this problem''. 
\end{quote}
Then they pose the following open question:

\begin{question}[M. Kazarian and B. Shapiro (1996), \cite{KS}]\label{QuestionKS} Does every closed parallelizable $m$-dimensional manifold admit a maximal growth distribution of rank $1 < n < m$?
\end{question}\label{shapiro}

They state the following related problem as well:

\begin{question}[M. Kazarian and B. Shapiro (1996), \cite{KS}]\label{QuestionKS2} When does a manifold admit a distribution whose associated flags have constant (and maximal-possible) ranks throughout the manifold?
\end{question}\label{shapiro}

We will address both questions in this article. Our main Theorem reads as follows.

\begin{theorem}\label{mainthm}
Let $M$ be a smooth manifold. The complete $C^0$-close $h$-principle holds for maximal growth $k$-distributions on $M$ if $k>2$.
\end{theorem}

This, together with Proposition \ref{ExistenceFormal} in Subsection \ref{FormalExistence}, implies the following \textit{existence} result:

\begin{theorem}\label{parallelizability}
Every parallelizable $n$-dimensional smooth manifold $M$ admits a $k$-distribution of maximal growth if $2<k\leq n$.
\end{theorem}

Theorem \ref{parallelizability} gives a positive answer to Question \ref{QuestionKS} by M. Kazarian and B. Shapiro for all values $2<n<m$ and Theorem \ref{mainthm} goes further by establishing a whole classification in terms of the formal data. 

Theorem \ref{mainthm} also implies Corollary \ref{ConstantRank} below, which  provides an answer for $k>2$ to Question \ref{QuestionKS2} in terms of an algebro-topological condition; i.e. in terms of existence of formal distributions.

\begin{corollary}\label{ConstantRank}
A smooth manifold admits a  rank $>2$ maximal growth distribution if and only if it admits a formal rank $>2$ distribution of maximal growth.
\end{corollary}

The following question, which we also tackle, was posed by A. del Pino in \cite{Pino2} regarding classification results of distributions in terms of the underlying topological/algebraic data:
\begin{question}[A. del Pino (2019), \cite{Pino2}] Given a class of distributions whose nilpotentisation is fibrewise isomorphic to some generic graded
Lie algebra $\mathfrak{g}$ (or a generic family of them), can we tell whether some classic h-principle technique
(say, convex integration) applies to provide a complete classification result, purely in terms of $\mathfrak{g}$?
\end{question}

Theorem \ref{mainthm} gives a positive answer to this question for $k>2$ and the case of formal nilpotentisations of maximal growth (see Definition \ref{MaxGrNilp}). For the case $k=2$, we provide a negative answer to the applicability of convex integration by showing that the associated differential relation, which we denote by $\mathcal{R}^{\step -r}\subset J^{r}\left(\Gr_2(TM)\right)$, fails to be ample, as the following theorem shows.

\begin{theorem}\label{NonAmplenessRank2}
Let $M$ be a smooth manifold of dimension $n$. The differential relation of maximal growth for rank-$2$ distributions $\mathcal{R}^{\step -r}\subset J^{r}\left(\Gr_2(TM)\right)$ is not ample in principal directions for any $n\geq 3$.
\end{theorem}

Note that non-ampleness of the Engel and the $(2,3,5)$-conditions follow from Theorem \ref{NonAmplenessRank2} as particular cases. Thus, some other techniques may be necessary in order to establish flexibility results for rank-$2$ distributions. Flexibility for the Engel case was explored by A. del Pino and T. Vogel in \cite{PV}, where they  showed that, analogously to the contact case, there exist overtwisted Engel classes abiding by a complete $h$-principle. See also \cite{CPP}.

Theorem \ref{parallelizability} also answers, for $k>2$, an open question raised during the workshop ``Engel Structures'' (see \cite{AIM2}) held in April 2017 at AIM (American Institute of Mathematics, San Jose, California). 

\begin{question}[AIM Problem List (2017), Problem 6.2 in \cite{AIM}]\label{questionAIM}
Are there examples of pairs $(n,k)$ with $n > k \geq 2$ such that for any parallelizable n manifold, there exists a $k$-plane field $\SD\subset TM $ with maximal growth vector? \footnote{We believe that the real intention of the question was to determine for which pairs $(k,n)$ there exist such distributions, since the mere existence of the pairs themselves is trivially answered, for example, by the well known result about existence of contact structures.}
\end{question}

Additionally, the following generalisation of the previous question was posed as well:

\begin{question}[AIM Problem List (2017), Remark to Problem 6.2 in \cite{AIM}]\label{remarkAIM} More generally, instead of parallelizability, one would like to assume only that one is dealing with manifolds admitting appropriate partial flags.
\end{question}

Question \ref{questionAIM} was further refined in \cite{MAP}, in line with the remark above, to ask whether any formal distribution can be homotoped inside the space of formal distributions to an actual maximal growth distribution:

\begin{question}[A. del Pino and J. M-A (2021), \cite{MAP}]\label{refinement} Does any formal distribution of maximal growth admit a holonomic representative up to homotopy?
\end{question}

In \cite{MAP} A. del Pino and the author showed that the result holds for $\step$-$2$, rank$>2$. The general case for distributions of rank and step greater than $2$ remained open and was suggested as an interesting open question, which Theorem \ref{mainthm} now solves. This article answers in the positive, for $k>2$, Question \ref{questionAIM}, Question \ref{remarkAIM} and Question \ref{refinement}.

Note that this is the optimal range for the rank of $\SD$ where one could expect a general $h$-principle statement to hold, since the contact $(2,3)$-case is well known not to abide by an $h$-principle \cite{Ben}.

The approach in this article is essentially different from the one in \cite{MAP}. There the problem was tackled from the point of view of differential forms, whereas here we deal with the description of distributions in terms of frames. This approach allows to apply the general higher-order version of convex integration and thus solve the general higher order case. Along the way, we establish a new criterion for checking ampleness of a differential relation (Subsection \ref{Red1}). This constitutes a result with interest on its own within the general theory of convex integration. 

It is worth noting that there has been some serious development of foundational nature regarding convex integration in recent years (\cite{T1, Th, MT, MAP}). P. Massot and M. Theillière showed in \cite{MT} that convex integration implies the holonomic approximation theorem at the level of order-$1$ jets, providing further evidence of the broad scope of this technique. See also \cite{GS} for a recent application of convex integration to the holomorphic setting, where local $h$-principles are shown to hold for complex Engel and complex even-contact structures.

\textbf{Acknowledgements:} I would like to thank Álvaro del Pino for his generosity sharing with me his knowledge about distributions, $h$-principles and more general topics in Differential Topology. His valuable comments and suggestions have greatly improved this article. I learnt about Lemma \ref{ExistenceLieAlgebraImpliesFormalDistribution} and its proof from him.

 I am grateful to Francisco Presas as well for introducing me into the study of convex integration and for his keen support and interest. I would also like to show my gratitude to Igor Zelenko for the fruitful discussions and his valuable knowledge. Finally, I want to thank the Geometry and Topology group at the University of the Basque Country where I was given a nice environment to develop this and other projects.

During part of the development of this work the author was supported by the ``Programa Predoctoral de Formación de Personal Investigador No Doctor'' scheme funded by the Basque Department of Education (``Departamento de Educación del Gobierno Vasco'').

\section{Maximal growth distributions} \label{ssec:distributions}

A distribution $\SD$ on a smooth manifold $M$ is a subbundle of the tangent bundle $TM$. By subsequently applying the Lie bracket to sections $\Gamma(\SD)$ of $\SD$, we get the following sequence of modules:
\[ \Gamma^1(\SD) \subset \Gamma^2(\SD) \subset \Gamma^3(\SD) \subset \cdots, \]
\[\text{ where } \Gamma^1(\SD) := \Gamma(\SD), \qquad \Gamma^{i+1}(\SD) := [\Gamma^1(\SD),\Gamma^i(\SD)], \]
 In the present article we will assume that all our distributions $\SD$ are  \emph{regular}, i.e. there is a distribution $\SD_i$ so that $\Gamma^i(\SD) = \Gamma(\SD_i)$. 

A key observation is that under these assumptions the flag 
\[ \SD_1 = \SD \subset \SD_2 \subset \SD_3 \subset \cdots \]
stabilises: i.e. there exists a natural number $r$ such that $\SD_i = \SD_{r}$ for all $i \geq r$. By Frobenius' Theorem this is equivalent to $\Gamma^{r}(\SD)$ being involutive and therefore $\SD_{r}$ being the tangent bundle of a foliation $\SF$ on $M$. 
We call the \textbf{Lie flag} associated to/produced by $\SD$ to the previous sequence.

\begin{definition}[Bracket-generating distribution of step $r$]\label{BrGenerating} If $\SD_{r} = \SF$, we say that that $\SD$ \textbf{bracket-generates} $\SF$ and if, moreover, $\SF=TM$ then we say that $\SD$ is \textbf{bracket-generating}. We call the first integer $r$ satisfying $\SD_r= TM$ the \textbf{step} of the distribution and we denote it by $\step(\SD)$.
\end{definition}

\begin{definition}[Growth vector]\label{gv}
Let $\mathfrak{n}_i=\dim(\SD_i)$. The vector $\nu_\SD=(\mathfrak{n}_1, \mathfrak{n}_2, \cdots, \mathfrak{n}_i, \cdots )$ is called the \textbf{growth vector} of $\SD$.
\end{definition}

\begin{remark} Since we are working under the assumption that all our distributions are regular; this definition depends solely on $\SD$ and not on any particular choice of point.
\end{remark} 

Let us discuss how we can establish a partial order in the set of growth vectors. 

\begin{definition}[Partial order on the set of growth vectors, \cite{AN}]\label{orderGV} We say that a distribution $\SD_1$ with growth-vector $\nu_{\SD_1}=(d_1, d_2, \cdots)$ grows faster than a distribution $\SD_2$  with growth vector $\nu_{\SD_2}=(\tilde{d}_1, \tilde{d}_2, \cdots)$ if $d_i\geq\tilde{d}_i$ for all $i\geq 1$ and, also, $d_j>\tilde{d}_j$ for some $j\geq 1$.
This defines a partial order on the set of growth vectors. 
\end{definition}

\begin{remark} Since we have defined a partial order we can talk about maximal elements in the space of growth vectors. There is, thus, a well defined notion of \textbf{maximal growth vectors} for regular distributions on $M$. This gives raise to the following definition.
\end{remark}

\begin{definition}[Maximal growth distribution]\label{MaxGrowthDistr}
We say that a distribution $\SD$ on a smooth manifold $M$ is a \textbf{maximal growth distribution} if its growth vector is maximal according to the partial order in Definition \ref{orderGV}.
\end{definition}

\begin{remark} For each dimension $n\in\mathbb{N}$ and rank $k\in\mathbb{N}$, the entries of a maximal growth vector can be computed explicitly. This is explained in Subsection \ref{subsection:Hall}.
\end{remark}

Let us fix some notation for the rest of the paper:
\begin{itemize}
\item We will use $\langle v_1,\cdots, v_n\rangle $ to denote the linear space spanned by vectors $v_1,\cdots, v_n$. Analogously, for a given set of vectors $A$, $\langle A\rangle$ will denote the linear space spanned by vectors in $A$.
\item Following \cite[1.3]{EM} and \cite{Gro86}, we will use $\Op(p)$ to denote an arbitrarily small but non explicitly specified open neighborhood of a point $p$ in a smooth manifold $M$.  
\end{itemize}

\subsection{The nilpotentisation}\label{sssec:nilpotentisation}

\begin{definition}[Nilpotentisation]\label{Nilpotentisation}
We define the \textbf{nilpotentisation} $\SL(\SD)$ of $\SD$ as the graded vector bundle
\[ \SD_1 \oplus \SD_2/\SD_1 \oplus \cdots \oplus \SD_i/\SD_{i-1} \oplus \cdots \oplus \SD_{r}/\SD_{r-1}. \]
\end{definition}

Note that the composition of taking the Lie bracket with the projection is $C^\infty$-linear

\[ \Gamma^j(\SD) \times \Gamma^i(\SD) \longrightarrow \Gamma^{i+j}(\SD) \longrightarrow  \Gamma^{i+j}(\SD)/\Gamma^{i+j-1}(\SD).\]

Therefore, it descends to the bilinear map

\[ \Omega_{i,j}(\SD): \SD_j/\SD_{j-1} \times \SD_i/\SD_{i-1} \longrightarrow \SD_{i+j}/\SD_{i+j-1} \]

that is called (i,j)-\textbf{curvature}.

All the curvatures together endow $\SL(\SD)$ with a fibrewise Lie bracket compatible with the grading. $\SL(\SD)$  is those endowed with a step-$r$ stratified Lie algebra structure (see Definition \ref{StratifiedLieAlgebra}).
 
Let us introduce some more notation. 

\begin{definition}\label{OrderedMI}
For a fixed integer $k$, define the set of ordered multi-indices $\mathbb{I}_\ell$ of order-$\ell$ as
\[
\mathbb{I}_\ell:=\bigl\{ (i_1,\cdots, i_\ell): 1\leq i_j\leq k\bigr\}.
\] 
Similarly, define $\mathbb{I}_0$ as the singleton  containing the empty index $O=()$. The set of ordered multi-indices of order $\leq i$ is defined as:
\[\mathfrak{I}_{i}=\bigcup_{\ell=0}^{i}\mathbb{I}_\ell. \]
\end{definition}

For a $k$-distribution $\SD$, fix a (possibly local) frame 

\[\mathfrak{Fr}=\lbrace X^1,\cdots, X^{k}\rbrace, \quad \SD=\langle X^1,\cdots, X^{k}\rangle.\]

For a given $I=(\ell_j,\cdots,\ell_1)\in\mathbb{I}_j$, we denote

\[A_I:=[X^{\ell_j}, \cdots, [X^{\ell_3}, [X^{\ell_{2}}, X^{\ell_1}]]\cdots]\]

We borrow the idea for such a compact notation from \cite{Lewis}. Note that if the length of $I$ is $1$, then the expresion $A_I$ denotes a single vector field. 
 
\begin{definition}\label{bracketsI}
For $i> 1$, we define $\mathfrak{Br}^{i}$ as the set of brackets of vector fields $X^1, X^1, \cdots, X^{k}$ (possibly with repetitions) in $\mathfrak{Fr}$ of length less or equal than $i$:

\[
\mathfrak{Br}^{i}:=\lbrace A_I: I\in \mathfrak{I}_i  \rbrace .
\]

\end{definition}

\begin{remark}
The frame $\mathfrak{Fr}$ itself coincides with the set $\mathfrak{Br}^{1}$ (where each element of the frame can be understood as a length-$1$ bracket).
\end{remark}

The following proposition states some characterisations of the bracket-generating condition with a fixed growth vector:

\begin{proposition}\label{characterisationsBG}
A $k$-distribution $\SD$ on a differentiable manifold $M$ has growth vector $\nu=(\mathfrak{n}_i)_{i=1}^r$,  if and only if any, and thus all, of the following equivalent conditions are satisfied:
\begin{itemize}
\item[i)] For $i=1,\cdots, r$, the $i$-th element $\SD_i$ of the Lie flag has dimension $\dim(\SD_i)=\mathfrak{n}_i$.
\item[ii)] All local frames $\mathfrak{Fr}$ of $\SD$ satisfy, for $i=1,\cdots, r$, $\dim\left(\langle \mathfrak{Br}^i\rangle\right)=\mathfrak{n}_i$.
\item[iii)] There exists a local frame $\mathfrak{Fr}$ of $\SD$ s.t. for $i=1,\cdots, r$, $\dim\left(\langle \mathfrak{Br}^i\rangle\right)=\mathfrak{n}_i$.
\end{itemize}

\end{proposition}

\begin{proof}
Condition $i)$ is just a rephrasing of the Definition of growth vector. The equivalences of $i)$ with conditions $ii)$ and $iii)$ readily follow from the fact that the rank of a bilinear map does not depend on the choice of basis.  
\end{proof}

\subsection{Free Lie algebras and Hall bases}\label{subsection:Hall}

For fixed $k$, the step $r$ of a maximal growth $k$-distribution depends on $n=\dim(M)$. But it is interesting to note that for a fixed step $r$, the first $r-1$ entries of the growth vector $\nu_\SD=(\mathfrak{n}_i)_{i=1}^r$ associated to a maximal growth $k$-rank distribution $\SD$  only depend on the integer $k$ and can be calculated explicitly.  The last entry $\mathfrak{n}_r$ equals the dimension of $M$.  Let us introduce some terminology prior to elaborating on this. 

\begin{definition}[Bracket expression, \cite{Pino2}]
We say that a string $x$, depending on the variable $x$, is a length-$1$ \textbf{bracket expression}. Analogously, the string $[x_0,x_1]$, depending on the variables $x_0$ and $x_1$, is a length-$2$ bracket expression. Inductively, we define a length-$n$ bracket expression to be a string of the form $[A(x_1,\cdots,x_j),B(x_{j+1},\cdots, x_n)]$ where $1 < j < n$ and where $A$ and $B$ are bracket expressions of lengths $j$ and $n-j$, respectively. We denote by $\ell(A)$ the length of $A$. 

Given a set $X=\lbrace x_1,\cdots, x_m\rbrace$, we denote by $M(X)$ the set of all possible bracket expressions of elements in $X$. Additionally, we denote by $M_i(X)$ the set of length-$i$ bracket expressions in $M(X)$.
\end{definition}
\begin{remark} From an abstract algebra point of view, the set $M(X)$ can be understood as the free magma generated by the set $X$.
\end{remark}
Let us fix some notation: the free algebra with $n$ generators is denoted by $\mathfrak{Lie}_n$, while we use the terminology $\mathfrak{Lie}_n^k$ for the linear subspace spanned by elements of length-$k$. 

We will now introduce the notion of a \em Hall basis \em associated to $\mathfrak{Lie}_n$. Bases for free Lie algebras appeared for the first time in the work of M. Hall \cite{Hall} and, since, many articles appeared describing such bases. We will rather work with the Definition from the book \cite[pp. 22--23]{Ser} by J. P. Serre.

\begin{definition}[Hall set associated to $\mathfrak{Lie}_n$ \cite{Ser}]\label{hall}

Consider an ordered set of $n$ generators 
\[X=\lbrace X_1<\cdots< X_n\rbrace.\] 

We say that a totally ordered subset $H\subset M(X)$ is a \textbf{Hall set} 
if the following conditions are satisfied:

\begin{itemize}
\item[i)] $X\subset H$,
\item[ii)] if $\ell(a)<\ell(b)$ for $a, b\in H$, then $a<b$; i.e. brackets in $H$ are ordered by length,
\item[iii)] $[a, b]\in H$ if and only if the following two conditions are satisfied:

\begin{itemize}
\item[iii.a)] $a, b\in H$ and $a<b$,  
\item[iii.b)] either $b\in X$ or $b=[c,d]$ where $c, d\in H$ and $c\leq a$.
\end{itemize}

\end{itemize}

We denote by $\mathcal{V}_i$ the subset of length-$i$ brackets in $H$; i.e. $\mathcal{V}_i:=H\cap M_i(X)$.

\end{definition}

We often refer to a Hall set produced by an ordered set $X$ with $n$-generators as a \textbf{Hall basis} associated to $X$ or, also, associated to $\mathfrak{Lie}_n$. This terminology is justified by the following well known proposition.

\begin{proposition}[Hall sets are graded bases \cite{Hall, Witt, Reu}]
A Hall set $H$ associated to $\mathfrak{Lie}_n$ constitutes a graded basis of $\mathfrak{Lie}_n$, whereas each subset $\mathcal{V}_k\subset H$ provides a basis for $\mathfrak{Lie}_n^k$. Moreover, the dimension $d_{n,k}$ of each $\mathfrak{Lie}_n^k$ is given by  the expression

\[d_{n,k}=\frac{1}{k}\sum_{p|(k)}\mu(p) n^{k/p},
\] 

where $\mu(\cdot)$ denotes the M\"obius function from Number Theory.
\end{proposition}

Note that from a given local frame $\mathfrak{Fr}$ of a maximal growth distribution $\SD$ we can produce a Hall basis associated to it  as follows. Let

\[\mathfrak{Fr}=\lbrace X_1,\cdots, X_k\rbrace,\quad \SD=\langle X_1, \cdots, X_k\rangle.\]

Note that $\SD_i=\langle \mathfrak{Br}^i\rangle$ for $1\leq i\leq r$. It turns out that length-$i$ brackets involving elements of the frame $\mathfrak{Fr}$ are generators of $\SD_i/\SD_{i-1}$ and, in the case that $i<r$, the only dependance relations among these brackets are the ones given by the antisymmetric property and Jacobi identity. Therefore, the first $r-1$ layers $\mathcal{V}_i$ in a Hall basis $H$ associated to the frame  $\mathfrak{Fr}$ constitutes a graded basis of the first $r-1$ elements of the nilpotentisation. In other words, each $\mathcal{V}_i\subset H$ ($i\neq r$) is a basis of $\SD_i/\SD_{i-1}$ and, so, $\dim(\SD_i/\SD_{i-1})=\dim(\mathcal{L}ie^i_k)=d_{k,i}$.

This is encoded by the following well known proposition. 
\begin{proposition}[\cite{AN}. Growth vector of a distribution of maximal growth]\label{gvMax}
The growth vector of a step-$r$, rank-$k$ distribution $\SD$ of maximal growth on a smooth manifold of dimension $n$ is given by

\[\nu_{\SD}=(\mathfrak{n}_1, \mathfrak{n}_2,\cdots, \mathfrak{n}_{r-1}, \mathfrak{n}_r)=\left(k=d_{k,1}, d_{k,1}+d_{k, 2}, \cdots, \sum_{i=1}^{r-1}d_{k, i}, n\right).\]

\end{proposition}

\begin{remark}\label{notationni}
Since the entries of a maximal growth vector depend solely on $k=\rank(\SD)$ and $n=\dim(M)$, we will often reserve the letters $\mathfrak{n}_i$ to denote the $i$-th entry of a maximal growth vector ($k$ and $n$ will be omitted from the notation whenever their values are clear from the context).
\end{remark}

\begin{definition}\label{FreeType}
We say that a $k$-rank maximal growth distribution $\SD$ is of \textbf{free type} if  the last entry of the growth vector satisfies the following equality:
\begin{equation*} \mathfrak{n}_r=\sum_{i=1}^{r}d_{k, i}.\end{equation*}

\end{definition}

\begin{remark}
The condition in Definition \ref{FreeType} tantamounts to the $r$ elements $\SD_i$ $(1\leq i\leq r)$ in the Lie flag of such a maximal growth distribution having the same dimensions as the first $r$ subspaces $\mathfrak{Lie}^i_k$  in the free Lie algebra $\mathfrak{Lie}_k$.
\end{remark}

\begin{example}
Some examples of maximal growth vectors of free type are $\nu_\SD=(2, 3, 5)$, $\nu_\SD=(2, 3, 5, 8)$, $\nu_\SD=(3,6, 14)$ and $\nu_\SD=(4, 10, 30)$; whereas $\nu_\SD=(3,6, 8)$ and $\nu_\SD=(4, 10, 11)$ are  examples of maximal growth vectors \em not \em of free type. \end{example}

Let us elaborate on a concrete example.
\begin{example}\label{ex:hall}
Consider a local frame $\mathfrak{Fr}=\lbrace X_1, X_2, X_3\rbrace$ and the maximal growth distribution it spans $\SD=\langle \mathfrak{Fr}\rangle$. We can explicitly give the elements in the Hall basis associated to $\mathfrak{Fr}$ following Definition \ref{hall}. Let's do this as an example for the three first elements in the flag:

\begin{align*} \SD_1= & \langle X_1, X_2, X_3\rangle,\\
\SD_2/\SD_1= &\langle [X_1, X_2], [X_1, X_3], [X_2, X_3]\rangle, \\
\SD_3/\SD_2= &\langle[X_1, [X_1, X_2]], [X_1, [X_1, X_3]], [X_2, [X_1, X_2]], [X_2, [X_1, X_3]], \\
 &[X_2, [X_2, X_3]], [X_3, [X_1, X_2]], [X_3, [X_1, X_3]], [X_3, [X_2, X_3]]\rangle.\end{align*}

\end{example}
Note that there are length-$3$ brackets not appearing in the Hall basis: $[X_1, [X_2, X_3]], [X_3, [X_2, X_1]], \cdots$ Indeed, they can be written as combinations of the ones in the Hall basis by taking into account Jacobi identity and the antisymmetric property. The growth vector for this distribution $\SD$ reads as $\nu_\SD=(3, 6, 14, \cdots)$. 

We will outline the following elementary fact as a Lemma since it will be useful later. 

\begin{lemma}\label{HallBasisElements}
Consider an ordered set $\lbrace X_1 < \cdots< X_k\rbrace$ and a Hall basis $H\subset M(X)$ associated to $X$ as in Definition \ref{hall}. Then, \begin{itemize}
\item if $j>i$, the length-$\ell$ bracket $[X_i, [X_i,[ \cdots, [X_i, [X_i, X_j]]\cdots]]]$ belongs to the Hall basis $H$.
\item If $j<i$, then the bracket $[X_i, [X_i,[ \cdots, [X_i, [X_j, X_i]]\cdots ]]]$ belongs to the Hall basis $H$.
\end{itemize}
\end{lemma}

\begin{proof}
First, note that any bracket formed by $X_i$, $X_j$, where the former appears $\ell-1$ times and the latter exactly $1$ time cannot  be expressed as a combination of any other bracket with any other $X_m$ involved, $m\neq i,j$, or $X_i, X_j$ appearing a different number of times. Indeed, the number of times each element appears in each bracket is preserved both by Jacobi identity and the antysimmetric property. So, such a bracket formed by $X_i, X_ j$ must be part of any Hall basis associated to $\mathfrak{Fr}$. By conditions in Definition \ref{hall} it follows that the only possibilities are the ones stated in this Lemma.
\end{proof}

\begin{remark}
Brackets $[X_i, [X_i, [\cdots, [X_i, [X_i, X_j]\cdots ]$ as in Lemma \ref{HallBasisElements} where $X_i$ appears $p$ times are often denoted by $ad_{X_i}^{p}(X_j)$ in the literature (see, for example, \cite[Thm. 3]{Ser}).
\end{remark}

\begin{example}
For $\ell=3$ in Example \ref{ex:hall}, brackets $[X_1, [X_1, X_2]], [X_1, [X_1, X_3]]$ are of such type for $i=1$,  $[X_2, [X_1, X_2]], [X_2, [X_2, X_3]]$ for $i=2$ and $ [X_3, [X_1, X_3]], [X_3, [X_2, X_3]]$ for $i=3$. 
\end{example}

\section{Convex integration}\label{convexintegration}

\subsection{Jet spaces: local description}\label{ssec:jetsCoord}

We will briefly recall some notation and notions about jet spaces. 

For a smooth fiber bundle $E\to M$, we denote by $J^r(E)$ the corresponding space of $r$-jets. For $s<r$ we write $\pi^r_s:J^r(E)\to J^s(E)$ for the corresponding projections, as well as $\pi^r_{bs}:J^r(E)\to E$ for the projections onto the base. 
We often use the notation $j^i(F)$ for denoting $\pi^r_i(F)\in J^r(E)$ if $i\leq r$.

Working in a local chart $\mathcal{U}$ of $E$, we can locally identify $M$ with $\R^n$ and $E$ with $\R^n\times\R^m$, where the fibers of the fibration identify with the $\R^m$ factor. By doing so, we have the following local description of the corresponding jet spaces:

\[E\supset\mathcal{U}\simeq J^r(\R^n\times \R^m)\simeq \R^n\times \R^m\times \Hom(\R^n, \R^m)\times\Sym^2(\R^n, \R^m)\times\cdots\times\Sym^r(\R^n, \R^m),\]

where $\Sym^d(\R^n, \R^m)$ denotes the space of symmetric homogenous polynomials of degree $d$ with entries in $\R^n$ and taking values in $\R^m$. Note that this identification works by identifying each jet $F\in J^r(\R^n\times \R^m)$ with its corresponding $r$-th order Taylor polynomial over the point $\pi^r_{bs}(F)\in \R^n$.

\begin{remark} It turns out that $\pi^1_0$ is a vector fibration whereas $\pi^r_{r-1}$ is in general an affine fibration that can be understood as the map that assigns to each order-$r$ Taylor polynomial its truncated order-$r-1$ part. 
\end{remark}

\subsection{Principal subspaces} \label{ssec:principalSubspaces}

The following notion formalises the idea of two $r$-jets that agree except along a pure derivative of order $r$:
\begin{definition} \label{def:perpJet}
Given a hyperplane $\tau\subset T_pM$, we say that two sections $f,g: M \to X$ have the same $\bot(\tau,r)$-jet at $p \in M$ if

\[ D_p|_\tau j^{r-1}f =  D_p|_\tau j^{r-1}g, \]

where $D_p|_\tau$ means taking the differential at $p$ and restricting it to $\tau$.
\end{definition}
When $\tau$ is a hyperplane field, the $\bot(\tau,r)$-jets form a bundle, which we denote by $ J^{\bot(\tau,r)}(X)$.
There are affine fibrations:

\[ \pi^r_{\bot(\tau,r)} : J^r(X) \rightarrow J^{\bot(\tau,r)}(X), \quad  \pi^{\bot(\tau,r)}_{r-1}: J^{\bot(\tau,r)}(X) \rightarrow J^{r-1}(X). \]

Given a section $f: M \to X$, we write 

\[ j^{\bot(\tau,r)}f: M \to J^{\bot(\tau,r)}(X) \]

for the corresponding section of $\bot(\tau,r)$-jets. A section of this form is called \textbf{holonomic}.

\begin{definition} \label{def:principalSubspaces}
The fibers of $\pi^r_{\bot(\tau,r)}$ are said to be the \textbf{principal subspaces} associated to $\tau$ (and $r$). They are all affine subspaces parallel to one another. Given $F \in J^r(X)$, we write

\[ \Pr^r_{\tau,F} := (\pi^r_{\bot(\tau,F)})^{-1}\left(\pi^r_{\bot(\tau,r)}(F)\right) \]

for the principal subspace that contains it. 
\end{definition}

\begin{remark}\label{IntegrableHyperplanes}
For the most common used flavours of convex integration (ampleness in principal coordinate directions, see \cite[1.1.4]{MAP}, and ampleness in principal directions, see Definition \ref{def:ampleness1}),  all the hyperplane fields $\lambda\in TM$  we work with integrate to a codimension-$1$ foliation. 

Then, when one passes to a local chart (and, thus, a choice of coordinates has been made), instead of talking about hyperplane fields, we will often identify the hyperplanes with the direction they define by duality (given by the Euclidean metric of $\R^n$ in the chart). We then write $\bot(\partial_i,r) := \bot(\ker(dx_i),r)$, where in local coordinates $\partial_i$ is the dual vector field of $dx_i$. 
\end{remark}

Working in local coordinates, and assuming $M=\R^n$ and $\tau:=\ker(dt)$, let $\R^n$ split as a product 
$\R^n=\R\times\R^{n-1}$  with coordinate $t$ for $\R$ and coordinates $x_2, \cdots, x_{n}$ for $\R^n$. The $r$-th order jet $j^r(f)$ of a section $f:\R^n \to X$ is determined by all the partial derivatives 

\[\lbrace{\partial^\alpha_X\partial^\beta_t f\rbrace},\quad \alpha=(\alpha_2,\cdots,\alpha_{n}), \quad |\alpha|=\alpha_2+\cdots+\alpha_{n}\leq r-\beta. \]

Here $\partial^\alpha_X f$ denotes the partial derivative of $f$ with respect to $x_2,\cdots, x_{n}$ with orders $\alpha=(\alpha_2,\cdots,\alpha_{n})$ for each direction, respectively. Analogously, $\partial^\beta_t f:={\partial^\beta f}/{\partial t^\beta}$.

Note that the space $J^{\bot(\tau, r)}$ (Definition \ref{def:perpJet}) corresponds to the set of equivalence classes of partial derivatives $\lbrace{\partial^\alpha_X\partial^\beta_t f\rbrace}$ for which $\beta\leq r-1$ (abusing notation and referring here to the $0$-jet part as an ``order-$0$ derivative''). Then the splitting of $\R^n$ induces an expliting of $j^r(f)$ as follows:

\[j^r(f)= j^{\bot(\tau, r)}f \oplus \partial^r_t f\]

which can be understood as a splitting in ``pure derivatives of order $r$ in the direction of $t$'' and ``the rest of mixed partial derivatives'' (this includes pure derivatives of order $r$ or less in some other directions as well).

This same splitting is thus inherited in the space of formal solutions as well:
\begin{equation}\label{FormalSplitting}
F=j^{\bot(\tau, r)}(F)\oplus j^r_t(F),
\end{equation}

where the component $j^r_t(F)$ represents \textbf{the order-$r$ pure formal derivative of $F$ with respect to $\partial_t$}.  Analogously, $j^{\bot(\tau, r)}(F)$ denotes the component corresponding to the rest of mixed formal partial derivatives.

\begin{remark}
This splitting depends on the splitting of $\R^n=\R[t]\times\R^{n-1}$ but not on the coordinates chosen within $\R^{n-1}$.
\end{remark}

Given a differential relation $\SR \subset J^r(X)$ of order $r$, the projection of $\SR$
\[ \pi^r_{i}(\SR)\subset J^{i}(X) \]

is a differential relation of order $i < r$. 
It is also open since the maps $\pi^r_{i}$ are submersive and thus open.

 Given $F \in J^r(E)$, we write 
\[ \SR^r_{\tau,F} := \SR\cap \Pr^r_{\tau,F} \]

for the restriction of $\SR$ to the principal subspace $\Pr^r_{\tau,F}$ containing $F$. 

\begin{remark}\label{NotationPr}
Sometimes we are interested in considering the analogous objects for the projection of a given relation $\SR$ and point $F$ to the space of lower order jets; i.e.  for $\pi^r_{i}(\SR)$ over $\pi^r_{i}(F)$, where $i<r$. So, for the ease of notation and whenever it is clear from the context, we will still denote $\SR$ for $\pi^r_{i}(\SR)$ and we identify $F$ with $\pi^r_{i}(F)$. This way,  we will often use the following notation:

\[\Pr^i_{\tau,F}:=\Pr^i_{\tau,\pi^r_{i}(F)}, \quad \SR^i_{\tau,F}:=\left(\pi^r_{i}(\SR)\right)^i_{\tau,\pi^r_{i}(F)}. \]

Similarly, using the notation $j^i(F)$ for denoting $\pi^r_i(F)$ (beginning of Subsection \ref{ssec:jetsCoord}), there are well defined decompositions:
\begin{equation}\label{FormalSplitting2}
j^i(F)= j^{\bot(\tau, i)}(F) \oplus j^i_t(F).
\end{equation}
\end{remark}

\subsection{Ampleness in principal directions} \label{ssec:ampleness1}

Ampleness of a differential relation is a key notion within the theory of convex integration. We define ampleness for subsets of affine spaces first. We later adapt this notion to differential relations in jet spaces.
\begin{definition} \label{def:ampleAffine}
Let $X$ be an affine space and $Y \subset X$ a subset. Given $y \in Y$ we write $Y_y$ for the path-connected component containing it. We say that $Y$ is \textbf{ample} if the convex hull $\Conv(Y,y) := \Conv(Y_y)$ of each $Y_y \subset Y$ is the whole of $X$.
We say that ampleness holds trivially if $Y_y$ equals the empty set or the total space.
\end{definition}

Let us provide an example of an ample set which we will state as a Lemma. For completeness we reproduce the proof from \cite{MAP} with minor changes.

\begin{lemma}\label{LemmaGL}
The space $GL(n)$ of non singular matrices of order $n\times n$ is ample inside the space $\mathcal{M}_{n\times n}$ of order $n\times n$-matrices if $n\geq 2$ and non-ample otherwise.
\end{lemma}

\begin{proof}
The space $GL(n)$ has two connected components: the space of positive determinant $n\times n$-matrices $GL^+(n)$ and the one of negative determinant $GL^-(n)$. In order to check that both  components are ample inside $\mathcal{M}_{n\times n}$ we proceed as follows. Note that any $M\in\mathcal{M}_{n\times n}$ can be decomposed as a convex combination of non-singular matrices. Indeed, for $\mu\notin Spec(M)\setminus\lbrace 0\rbrace$, we have that $M=\frac{1}{2}\cdot 2\left(M-\mu\cdot Id\right)+\frac{1}{2}\cdot 2 Id$.

So, we just need to check that any matrix $M\in GL^+(n)$ (alternatively, in $GL^-(n)$) can be expressed as a convex combination of matrices in $GL^-(n)$ (alternatively, in $GL^+(n)$).  Writing the matrix by columns  $M=(v_1,\cdots, v_n)$ we then define, for $\varepsilon>0$:
\[
M_1=\left( (2+\varepsilon)\cdot v_1, -\varepsilon\cdot v_2, \cdots v_n\right),\quad M_2=\left( -\varepsilon\cdot v_1, (2+\varepsilon)\cdot v_2, \cdots v_n\right).
\]

Note that $M=\frac{1}{2}M_1+ \frac{1}{2}M_2$. Also,  both $M_1$ and $M_2$ are non-singular and do not belong to the same connected component as $M$, thus yielding the claim. 
\end{proof}

\begin{example}\label{hyperplane}
A classical example of a non-ample subset of a real affine space $\mathbb{A}$ is the complement of a hyperplane, $\mathbb{A}\setminus H$. Note that this set has two connected components, each of which is a convex set that coincides with its convex hull and, therefore, cannot be ample.
\end{example}

We now introduce the notion of ampleness along principal directions for differential relations. There are some other more general notions of ampleness (the reader can check \cite{MAP} for a comprehensive discussion about the different incarnations of ampleness) but these will not be relevant for our purposes.

\begin{definition} \label{def:ampleness1}
Take a bundle $X \to M$ and a differential relation $\SR \subset J^r(X)$. Take a direction $\lambda \in T_pM$. We say that $\SR$ is
 \textbf{ample along the principal direction} (determined by) $\lambda$ if, for every $F \in \SR$ projecting to $p$, $ \SR_{\lambda,F} \subset \Pr_{\lambda,F}$ is ample. 

More generally,  if the relations $(\pi^r_{r'}(\SR))_{r' = 1,\cdots,r}$ are ample along all non-zero directions $\lambda$, then we say that $\SR$ is \textbf{ample in principal directions}.
\end{definition}
This notion of ampleness is the most common one and when it is satisfied we sometimes just say that $\SR$ is \textbf{ample}.

 Gromov's Convex Integration Theorem was first proved for first order differential relations in \cite[Corollary 1.3.2]{Gr73} and later on for higher order differential relations in \cite[Section 2.4, p. 180]{Gro86}. Although it can be stated in more generality, we will state the following version which is enough for our purposes:

\begin{theorem}[Convex integration] \label{thm:convexIntegration1}
The complete $C^0$-close $h$-principle holds for any open relation that is ample in all principal directions.
\end{theorem}

Finally, we finish this Section with the following Lemma that establishes a general situation where ampleness holds and thus Thereom \ref{thm:convexIntegration1} applies.

\begin{lemma}\cite[p.173, Corollary (E)]{Gro86}\label{thinSubsets}
Let $\Sigma\subset J^r(X)$  be a a stratified subset of codimension $\geq 2$ such that the intersection of $\Sigma$ with every principal subspace has codimension $\geq2$ within the principal subspace. Then $J^r(X)\setminus\Sigma$ is an ample differential relation.
\end{lemma}

\begin{remark}
Subsets $\Sigma$ as in the previous Lemma are called \textbf{thin singularities} or thin subsets. Instances of such subsets arise, for example, in the proofs of the $h$-principle for even-contact structures \cite{McD} or the $h$-principle for real and co-real immersions \cite{Gro86}.
\end{remark}

\begin{remark}\label{NonThinnesOfGL}
Note that the complement of $GL(n)$ inside the space $\mathcal{M}_{n\times n}$ is the space of matrices with zero determinant (determined by the equation $\det(A)=0$). Therefore, the associated singularity has codimension-$1$ and is, thus, not thin. Lemma \ref{LemmaGL} provides, thus, an example of an ample set not having a thin singularity.
\end{remark}

\section{Formal bracket-generating distributions of step-$r$}\label{FBD}

\subsection{Jet-coordinates}\label{jetcoordinates}

We will introduce jet-coordinates so that we can work in a comfortable and transparent way with expressions in jet spaces. We will focus our attention on $J^{r-1}\left(\bigoplus_k T\R^n\right)$ since this is the main jet space we will work with.  We follow, adapting it to our case, the notation and elegant exposition from \cite{Olver}.
\begin{definition}
Define the set of unordered multi-indices $\mathbb{S}_\ell$ of order-$\ell$ as:
\[
\mathbb{S}_\ell:=\mathbb{I}_\ell/\Sigma_\ell,
\] 

(recall Definition \ref{OrderedMI}) where the quotient by the symmetric group $\Sigma_\ell$ just means that two multi-indices are equal if they coincide up to reordering their elements. Similarly, define $\NS_0$ as the singleton containing the empty index $O=()$. The set of non-ordered multi-indices of length less or equal than $i$ is defined as:
\[\mathcal{S}_{i}=\bigcup_{\ell=0}^{i}\mathbb{S}_\ell. \]
\end{definition}

For $F=(F_1,\cdots, F_k)\in J^{r-1}\left(\bigoplus_k T\R^n\right)$ over the point $x\in\R^n$, we write $j^0(F_i)=(u^1_i,\cdots, u_i^n)$. We think the $u_i^j$ as smooth maps depending on $x=(x_1,\cdots, x_n)$. This way, we write 

\begin{equation}\label{CoordinatesF}
F=(x, u^1_1, \cdots, u^j_i \cdots, u^n_k, \cdots, u^j_{i,I}, \cdots), \quad x\in\R^n, \quad i=1,\cdots, k, \quad j=1,\cdots, n, \quad I\in\mathcal{S}_{r-1}
\end{equation}

where $u^j_i$ denotes the $j$-$th$ component of $j^0(F_i)$. Similarly, for $I\in\NS_\ell$, the coordinate $u^j_{i, I}$ identifies with the partial derivative ${\partial^I u^j_i}/{\partial x^I}$, where we treat the jet variables $u^j_i$ as smooth maps. We often use as well the following compact notation: $u_{i, I}:=(u^1_{i,I},\cdots, u^n_{i,I})$ regarding it as a vector in $\R^n$.

Following the same terminology as \cite{Olver}, we introduce the following notion.

\begin{definition}\label{FWD}
A \textbf{differential polynomial} is an $n$-dimensional expression $P(\cdot)=\left(P_1(\cdot),\cdots, P_n(\cdot)\right)$ or, alternatively, $P(\cdot)=\sum_{i=1}^n P_i(\cdot) \partial_i$, where each $P_i(\cdot)$ is a polynomial function in the $u^j_{i,I}$ jet-coordinates,
\[
P(F)=P(u^1_1, \cdots, u^j_i \cdots, u^n_n, \cdots, u^j_{i,I}, \cdots).
\]
The \textbf{order} of $P$ is defined as the maximum order $I$ among all the individual jet-coordinates $u^j_{i,I}$ appearing in the expression. We denote by $\mathcal{P}_{\ell}$ the set of differential polynomials of order less or equal than $\ell$.
 \end{definition}

\begin{remark}\label{OrderWellDefined}
Since we are working under the framework of jets in  $J^{r-1}\left(\bigoplus_k T\R^n\right)$,  $\mathcal{P}_{\ell}$ is not defined for $\ell> r-1$.
\end{remark}

\begin{remark}\label{DiffPolAreSymbols} Upon regarding the jet-coordinates as smooth maps depending on the variables $x_1,\cdots, x_n$, a differential polynomial $P(\cdot)$ of order $m$ can be thought as the symbol $P_\symb: J^{r-1}(\oplus_k T\R^n)\to J^0(\oplus_k T\R^n)$  of some differential operator. \end{remark}

For each coordinate direction $\partial_t\in\R^n$, we define the following operation within the space of differential polynomials.
\begin{definition}\label{DerivationFormal}
A \textbf{directional derivation} $D_t$ is an operation defined in $\mathcal{P}_{r-2}$ which is linear, satisfies Leibniz rule and is defined as follows on individual coordinates:
\begin{equation}
\quad D_t(u_{i,J}^j)=u_{i,(J,t)}^j, \quad J\in\mathcal{S}_{r-2},
\end{equation}

where for $J=(i_1,\cdots, i_\ell)$ the notation $(J,t)$ just denotes the multi-index concatenation $(J,t):=(i_1,\cdots, i_\ell, t)$. By regarding the variables $u_{i, J}^j$ as smooth maps, by linearity and Leibniz rule, the definition of $D_t$ readily extends to the whole $\mathcal{P}_{r-2}$. 
\end{definition}

\begin{remark}
 Note that $D_t$ increases the order of a differential polynomial by one. Therefore, since we are working under the framework of $J^{r-1}\left(\bigoplus_k T\R^n\right)$, the assumption on the order of the differential polynomials (being less than $r-2$) is essential in order Definition \ref{DerivationFormal} to make sense. Recall Remark \ref{OrderWellDefined}.
\end{remark}

\begin{remark}\label{DtSymbol}
Note that $D_t$ just describes the symbol of the usual directional derivative with respect to the coordinate direction $\partial_t$ in $\R^n$, $n\geq 1$.
\end{remark}

Directional derivatives can be composed. For $J\in\NS_s$, the composition of derivatives 

\begin{equation} 
D_{(m_1,\cdots, m_\ell)}(u_{i,J}^j):=D_{m_\ell}\circ\cdots\circ D_{m_1}(u_{i,J}^j)
\end{equation}

 is well defined whenever $s+\ell\leq r-1$. We say that composed derivatives $D_I$ have order $k$  if $I\in\NS_k$.

\begin{remark}
Note that we use unordered subindices $I\in\NS_k$ to denote composed derivations. This reflects the commutativity of their composition.
\end{remark}

Next Lemma is an obvious remark that follows from the way jet-coordinates (\ref{CoordinatesF}) were defined.

\begin{lemma}\label{SplittingDt}
The splitting $F_m=j^{\bot(\tau, r-1)}(F_m)\oplus j^{r-1}_t(F_m)$ from Equation (\ref{FormalSplitting}) (or, equivalently, its analogous formula for lower order jets (\ref{FormalSplitting2})) can be described as follows in jet-coordinates:
\begin{itemize}
\item[i)] $j^{r-1}_t(F_m)$ corresponds to the order-$r-1$ pure formal derivative of $F$ with respect to $\partial_t$ and, thus, in jet-coordinates this component is described by 

\begin{equation*} u_{m, I}= \sum_{j=1}^n u^j_{m, I}\cdot \partial_j, \quad I=(t,\cdots, t)\in\NS_{r-1}.
\end{equation*}

This motivates Definition \ref{FormalDT}.

\item[ii)] Similarly, $j^{\bot(\tau, r)}(F_i)$ corresponds to the rest of mixed formal partial derivatives and the $0$-jets; i.e. its components are of the form $u_{(m, J)}$ with $J\in\mathcal{S}_{r-1}$ and $J\neq (t,\cdots, t)\in\NS_{r-1}$.
\end{itemize}
\end{lemma}

\begin{definition}\label{FormalDT} We denote by $P_t^i(\cdot)$ the map that assigns to each jet $F=(F_1,\cdots, F_k)\in J^{r-1}\left(\bigoplus_k T\R^n\right)$ its formal order-$i$ pure derivative in the $\partial_t$-direction:

\begin{equation}
P_{t}^i(\cdot):J^{r-1}\left(\oplus_k T\R^n\right)\to J^0(T\R^n), \quad F\mapsto (u^1_{m,I},\cdots, u^n_{m,I}), \quad I=(t,\cdots, t)\in\NS_{i}.
\end{equation}
\end{definition}

\begin{remark}\label{PtiSymbol}
Note that $P_t^i(\cdot)$  in Definition \ref{FormalDT} just describes, in line with Remark \ref{DtSymbol}, the symbol of $\frac{\partial^i}{\partial t^i}$, the order-$i$ derivative operator with respect to $\partial_t$.
\end{remark}

\subsection{Formal Lie brackets}

Think of vector fields in $\R^n$ as sections $X_j:\R^n\to T\R^n$. Given $k$ vector fields $X_1, \cdots, X_k\in\Sec(T\R^n)$, the usual Lie bracket of $X_i$ and $X_j$, which we denote by $A_{(i,j)}:=[X_i,X_j]\in\Sec(\R^n)$  is a new vector field produced out of the two.

 In other words, each lie bracket $A_{(i,j)}(X_1,\cdots, X_k)$ can be interpreted as a first order differential operator; $[\cdot, \cdot]:\Sec(\oplus_k T\R^n)\to\Sec(T\R^n)$. We can thus describe its symbol as a fibrewise map $[\cdot, \cdot]_{\symb}:J^1(\oplus_k T\R^n, \R^n)\to J^0(T\R^n)$.

Similarly, higher length Lie brackets $A(\cdot, \cdots, \cdot)$ can be interpreted as differential operators of higher order $A(\cdot, \cdots, \cdot):\Sec(\oplus_k T\R^n)\to\Sec(T\R^n)$ and their symbol as a map $A(\cdot, \cdots, \cdot)_{\symb}:J^{\ell-1}(\oplus_\ell T\R^n, \R^n)\to J^0(T\R^n)$, where $\ell$ is the length of the bracket expression.

\begin{remark}
For practical reasons, since there are well defined projections $J^{r-1}(T\R^n)\to J^{\ell-1}(T\R^n)$ for $\ell-1<r-1$, henceforth we will regard the symbols of all these operators as the obvious lifted fibrewise maps $A(\cdot, \cdots, \cdot)_{\symb}:J^{r-1}(\oplus_\ell T\R^n, \R^n)\to J^0(T\R^n)$; i.e. we will consider  $J^{r-1}(\oplus_\ell T\R^n, \R^n)$ as the common domain for all of them.
\end{remark}

From now on, for clarity and since we will mainly work with the symbols rather than with the operators themselves, we will drop the subindex $\symb$ from the notation. We will refer to the symbols of the various lie bracket operators as \textbf{formal Lie brackets} but, again, we will often drop the word ``formal'' whenever it is clear from the context.

Next Lemma describes the symbols of Lie brackets in jet-coordinates (we follow the same notation as in Expression (\ref{CoordinatesF}) for the coordinates of a generic jet $F=(F_1,\cdots, F_k)\in J^{r-1}\left(\bigoplus_k T\R^n\right)$. 

\begin{lemma}\label{LemmaSymbolBrackets} 

A \textbf{formal Lie bracket} of length-$\ell$ $(\ell\leq r)$ is a map 
\[ 
[\cdot,[\cdots, [\cdot,\cdot]\cdots]]: J^{r-1}\left(\bigoplus_k T\R^n\right)\to J^{0}\left( T\R^n\right).
\]
described inductively on the length. For $F=(F_1,\cdots, F_k)\in J^{r-1}\left(\bigoplus_k T\R^n\right)$, length-$\ell$  formal brackets $[F_{a_\ell}, [ \cdots, [F_{a_2}, F_{a_1}]\cdots]]$, $ (a_i=1,\cdots k)$, are described as follows in terms of jet-coordinates.

 Length-$1$ formal brackets $[F_{a_1}]$ correspond to the components $u_{a_1}$:
\begin{equation}\label{BracketCoord1} 
\left[F_{a_1}\right]:= \sum_{i=1}^n u^i_{a_1}\partial_i
\end{equation}
 
We can write length $\ell-1$ brackets as $[F_{a_{\ell-1}}, [ \cdots, [F_{a_2}, F_{a_1}]\cdots]]=\sum_{i=1}^n p^i \partial_i$, where $p=(p^1,\cdots, p^n)$ is a differential polynomial in the jet-coordinates of $F$. Length-$\ell$ brackets then read as:

\begin{equation}\label{BracketCoord2}
\left[F_{a_\ell}, \left[F_{a_{\ell-1}}, \left[ \cdots, \left[F_{a_2}, F_{a_1}\right]\cdots\right]\right]\right]:=\sum_{i,j=1}^n \left( u_{a_\ell}^j D_j (p^i) - p^j D_j(u^i_{a_\ell}) \right)\partial_i.
\end{equation}

\end{lemma}

\begin{remark}
For the ease of notation, from now on we will often denote along the article a bracket of the form $\left[F_{b_1}, \left[F_{a_{b}}, \left[ \cdots, \left[F_{b_{\ell-1}}, F_{b_\ell}\right]\cdots\right]\right]\right]$ as $A_{(b_1,\cdots, b_\ell)}(F)$. For example, the bracket in Equation (\ref{BracketCoord2}) can be expressed as $A_{(a_\ell,\cdots, a_1)}(F)$.
\end{remark}

\begin{remark}
Note that Lemma \ref{LemmaSymbolBrackets}  is just a rephrasing, in terms of jet-coordiantes, of the well known formulas for Lie brackets. As such, we can recover usual properties of the Lie bracket, as Remark \ref{multilinearity} and Lemma \ref{antisymmetry} illustrate. Equivalently, both readily follow from Lemma \ref{LemmaSymbolBrackets}.
\end{remark}

\begin{remark}\label{multilinearity}
Each formal bracket is multilinear in its entries.
\end{remark}

\begin{lemma}\label{antisymmetry}
For any multi-index $I=(a_\ell,\cdots, a_2, a_1)$, $\ell\geq 2$, consider the multi-index $\tilde{I}=(a_\ell,\cdots, a_1, a_2)$ defined by reversing the last two entries of $I$. Then the following equality holds:

\begin{equation}\label{AntSym} A_I(F)=-A_{\tilde{I}}(F)\end{equation}
\end{lemma}
\begin{proof}
One could argue that this is true for usual brackets and, being a formal bracket simply its symbol, then this condition readily follows. Equivalently, it follows from Lemma \ref{LemmaSymbolBrackets} as well. Note that Condition (\ref{AntSym})  for length-$2$ brackets is equivalent to antisymmetry, which readily follows from Lemma \ref{LemmaSymbolBrackets}. 

For higher order length-brackets, (\ref{AntSym}) follows in the same manner: just note that the iterative process described by Equation (\ref{BracketCoord2}) in Lemma \ref{LemmaSymbolBrackets} implies that Lie bracket expressions are antisymmetric in their last two entries. Thus exchanging the positions of $a_2$ and $a_1$ produces a change of sign in the whole expression. \end{proof}

\begin{definition}\label{SimplifiedForm}
We say that a legth-$\ell$ formal bracket $A_I(F)$, $(I\in\mathbb{I}_\ell)$, is expressed in \textbf{simplified form} if it is written as an expression involving solely non-zero order-$0$ jet-coordinates $u^j_{i}$ and order order $\leq\ell-1$ derivations of those, $D_J(u^j_i), J\in\mathcal{S}_{\ell-1}$.
\end{definition}

\begin{example} Length-$2$ formal brackets $[F_a, F_b]$ in simplified form read as:
\begin{equation}\label{BracketL2} 
\left[F_{a_2}, F_{a_1}\right]:= \sum_{i,j=1}^n \left( u_{a_2}^j D_j (u^i_{a_1}) - u^j_{a_1} D_j (u_{a_2}^i) \right)\partial_i
\end{equation}

Indeed, Expression (\ref{BracketL2}) only involves order-$0$ jet-coordinates $ u_{a_1}^j$ and $ u_{a_2}^j$ and order-$1$ derivations of those. An example of the same bracket not expressed in simplified form would just consist on replacing each $D_j(u_{a_\alpha}^i)$ by order-$1$ jet-coordinates $u_{a_\alpha,(j)}^i$.
\end{example}

\begin{lemma}\label{ExistenceSimplifiedForm}
Each length-$\ell$ bracket $A_{(a_1,\cdots, a_\ell)}(F)$ admits a simplified form for $\ell\geq1$.
\end{lemma}
\begin{proof}
It follows from Lemma \ref{LemmaSymbolBrackets} since the Lemma itself describes an iterative process that allows to decompose any length-$\ell$ bracket as a polynomial expression involving only order-$0$ jet-coordinates $u_i^j$ of  $F$ together with successive derivations $D_{j_m}\circ\cdots\circ D_{j_1}(u_i)$ up to order $\ell-1$.
\end{proof}

\begin{lemma}\label{DependanceJet}
The expression $A_{(a_1,\cdots, a_\ell)}(F)$ depends solely on the $\ell-1$-order information of the input $F$; i.e. $A_{(a_1,\cdots, a_\ell)}(\cdot)$ factors through some map $h$ as follows:
\[
A_{(a_1,\cdots, a_\ell)}(\cdot)=h\circ j^{\ell-1}(\cdot).\] 
\end{lemma}
\begin{proof}
Since $A_\ell(F_{a_1}, \cdots, F_{a_\ell})$ admits a simplified form (Lemma \ref{ExistenceSimplifiedForm}), it can be defined by the order-$0$ jet-coordinates $u_i^j$ of $F$ and derivations of those up to order $\ell-1$, thus yielding the claim.
\end{proof}

\subsection{Formal Lie flags}

Note that since we have well defined formal brackets of jets, we can now define the analogs of the sets $\mathfrak{Br}^{i}$ (Definition \ref{bracketsI}) in the formal setting. 

\begin{definition}\label{BrFormal}
Let Let $F=(F_i)_{i=1}^k\in J^{r-1}(\oplus_k T\R^n)$. For $i> 1$, we define $\mathfrak{Br}_F^{i}$ as the set of brackets of $F$ of length less or equal than $i$ $(i\leq r)$:
\[
\mathfrak{Br}_F^{i}:=\lbrace A_I(F): I\in \mathfrak{I}_i  \rbrace .
\]
\end{definition}

Similarly, we can now define the analog notion of Lie flags in the formal setting as well:

\begin{definition}\label{FormalDi}
Let $F=(F_i)_{i=1}^k\in J^{r-1}(\oplus_k T\R^n)$. For each $i=1,\cdots, r$, we define the planes

\[
\SD_i(F):=\langle \mathfrak{Br}_F^i\rangle.
\]

We call the \textbf{formal Lie flag} associated to/produced by $F$ to the following flag of inclusions:

\[
\SD_1(F)\subset \SD_2(F)\subset\cdots\subset \SD_r(F).
\] 

Similarly, we say that $\SD_F$ has (formal) \textbf{growth vector} $(n_i)_{i=1}^r:=\left(\dim(\SD_i(F)\right)_{i=1}^r$.
\end{definition}

\begin{remark}
Take $p\in\R^n$ and a distribution $\SD\subset T\R^n$. Take a local frame $\mathfrak{Fr}=(X_1,\cdots, X_k)$ of $\SD$ over $\Op(p)$ and consider its associated sets $\mathfrak{Br}^i$. Similarly, consider the jet $F=(F_i)_{i=1}^k\in J^{r-1}(\oplus_k T\R^n)$ associated to the frame; i.e. where $F_i=j^{r-1}(X_i) (i=1,\cdots, k)$. Note that the sets $\mathfrak{Br}^i$ associated to $\mathfrak{Fr}$ coincide with its formal analogs $\mathfrak{Br}_F^i$ over each point in $\Op(p)$.

Consequently, the Lie flag associated to $\SD$ coincides with the formal Lie flag associated to $F$; i.e. for each $i=1,\cdots, r$, $\SD_i=\SD_i(F)$ over each point in $\Op(p)$.
\end{remark}

\begin{remark}\label{ommitF}
For the ease of notation, and whenever it is clear from the context, we will drop the letter $F$ from the notation of the just defined notions; i.e. we will just write $\mathfrak{Br}^{i}$ for $\mathfrak{Br}_F^{i}$ and $\SD_i$ for $\SD_i(F)$. Similarly, we often write $\SD_F$, or plainly $\SD$, for $\SD_1(F)$.
\end{remark}

Recall that usual Lie bracket operators behave well with respect to changes of coordinates. Since formal brackets are just the symbols of the usual bracket operators, this readily translates to the formal setting. We phrase this fact as a Lemma.

\begin{lemma}\label{InvarianceCoordinatesBrackets}
Let $F=(F_i)_{i=1}^k\in J^{r-1}(\oplus_k T\R^n)$ and let $f\in \Diff(\R^n, \R^n)$. Consider then the pushforward $f_*F=(f_*F_i)_{i=1}^k$. Then we have
\[ 
f_*A_{(a_1,\cdots, a_\ell)}(F)=A_{(a_1,\cdots, a_\ell)}(f_*F)\]

\end{lemma}

\begin{remark}
Lemma \ref{InvarianceCoordinatesBrackets} implies that the definition of the elements $\SD_i(F)$ does not depend on the choice of coordinates for $\R^n$. Therefore, we can extend the definition of the Lie flag (Definition \ref{FormalDi}) to arbitrary smooth manifolds. 

Indeed, we can just define it in local charts $U\subset M$ and, by the invariance under changes of coordinates (Lemma \ref{InvarianceCoordinatesBrackets}), this definition does not depend on the choice of chart. In other words, these objects are intrinsically well defined for any element $F\in J^{r-1}(\oplus_k TM)$.
\end{remark}

\subsection{Formal distributions}

\begin{definition}
 We define the bundle of $k$-frames over $M$ as the subspace of linearly independent $k$-tuples inside the space of $k$-tuples $\bigoplus_k TM$. We denote this bundle by $\overline{\bigoplus_k TM}\to M$.
\end{definition}

Note that there is a natural bundle projection that maps a $k$-tuple of pointwise linearly independent vector fields $(X_1,\cdots, X_k)$ (also called a $k$-frame) to the linear distribution they span:

\[ \pi: \overline{\bigoplus_k TM} \longrightarrow \Gr_k(TM), \]

which induces a map at the level of $r-$jets:

\[ j^r\pi: J^r\left(\overline{\bigoplus_k TM}\right) \longrightarrow J^r(\Gr_k(TM)). \]

\begin{remark}
We will often denote the map $j^r\pi$ by $\pi^r$.
\end{remark}

\begin{definition}
Given an $r$-tuple of integers of the form $\nu=(\mathfrak{n}_1=k, \mathfrak{n}_2, \cdots, \mathfrak{n}_r=\dim M)$, the differential relation 
\[\mathcal{R}^\nu \subset J^{r-1}\left(\overline{\bigoplus_k TM}\right)\]

  is defined by the set of elements ${F}\in J^{r-1}(\overline{\bigoplus_k TM})$ for which $\dim (\SD_{i}(F))=\mathfrak{n}_i$ for every $1\leq i\leq r$. We say that $\mathcal{S}^\nu$ is the differential relation of \textbf{formal bracket-generating frames} with growth vector $\nu$.
\end{definition}

\begin{definition}\label{formalBG}
In the particular case of $\nu=(\mathfrak{n}_1, \cdots, \mathfrak{n}_{r-1}, \mathfrak{n}_r=\dim M)$ being the maximal growth vector corresponding to the ambient manifold $M$ (where $\mathfrak{n}_{r-1}< \dim M$), sections to the differential relation $\mathcal{S}^\nu$  are called formal frames of \textbf{maximal growth} of step-$r$ and we denote this differential relation by $\mathcal{S}^{\text{step-}r}$.
\end{definition}

\begin{definition}\label{Srel}
A jet ${F}\in J^{r-1}(\Gr_k TM)$ is called \textbf{formally bracket-generating of step-$r$} if $F$ is the projection of some element $\tilde{F}\in \mathcal{S}^{\text{step-}r}$; i.e. $\pi^{r-1}(\tilde{F})=F$. The subset of those elements $F$ will be denoted by $\mathcal{R}^{\text{step-}r}\subset J^{r-1}\left(\Gr_k TM\right)$.
\end{definition}

\begin{definition}\label{FDstepr}
We say that a \textbf{formal distribution} of step-$r$ on $M$ is a smooth section $s:M\to\mathcal{R}^{\text{step-}r}$.
\end{definition}

\begin{remark}
Note that the set of (germs of) holonomic sections to $\mathcal{R}^{\text{step-}r}$ coincides with the set of (germs of) maximal growth distributions on $M$ (recall Definition \ref{MaxGrowthDistr}).
\end{remark} 

\begin{remark}
The terminology $\mathcal{R}^{\text{step-}2}$ was introduced in \cite{MAP} to denote the differential relation defined by distributions of $\step$-2, which are always of maximal growth. We then use the notation $\mathcal{R}^{\text{step-}r}$ to denote the differential relation defined by $\step$-$r$ maximal growth distributions, naturally extending the terminology from \cite{MAP}. 
\end{remark}

\begin{lemma}\label{independenceLieFlagFrame}
Consider jets $G, \tilde{G}\in \mathcal{S}^{\text{step-}r}$ over a point $p\in M$ so that $\pi^r(G)=\pi^r(\tilde G)$. Then the Lie flag associated to $G$ coincides with the one associated to $\tilde{G}$.
\end{lemma}

\begin{proof}
Any jet $F \in J^{r-1}(\overline{\bigoplus_k TM})$ over $p$ can be regarded as the equivalence class of sections which coincide up to the $r-1$-order information over the point $p$. Take two sections $X=\lbrace X_1,\cdots, X_k\rbrace: M \to TM$ and $\tilde X=\lbrace \tilde{X}_1,\cdots, \tilde{X}_k\rbrace: M \to TM$ realizing, over $p$, the jets $G=(G_i)_{i=1}^k$ and $\tilde G=(\tilde G)_{i=1}^k$, respectively; i.e. so that
\[j^{r-1}(X)=G \quad\text{and}\quad j^{r-1}(\tilde{X})=\tilde{G}\quad\text{over }p.\]

It is well known that the Lie flag of a smooth distribution does not depend on the choice of frame. Therefore, since $X$ and $\tilde{X}$ coincide up to order-$r-1$ at $p$ (because $G=\tilde{G}$), then they have the same Lie flag over $p$. But since $j^{r-1}(X)=G$, the Lie flag produced by $j^{r-1}(G)$ over $p$ coincides with the one produced by $G$ (and so is the case for $j^{r-1}(\tilde X)$ and $\tilde G$, respectively).

 Putting everything together, we conclude that the Lie flag produced by $G$ coincides with the one produced by $\tilde{G}$, yielding the claim.
\end{proof}

\begin{remark}
Lemma \ref{independenceLieFlagFrame} implies that any given $F\in\mathcal{R}^{\text{step-}r}$ has a well defined associated Lie flag. Indeed, any two lifts $G, \tilde{G}\in\mathcal{S}^{\text{step-}r}$ with $\pi^r(G)=\pi^r(\tilde G)=F$ define the same formal Lie flag by Lemma \ref{independenceLieFlagFrame}. This gives raise to the following Definition.
\end{remark}

\begin{definition}
We define the (formal) \textbf{Lie flag} associated to $F\in\mathcal{R}^{\text{step-}r}$ as the Lie flag associated to any of its lifts; i.e. as the Lie flag associated to any $\tilde{F}\in\mathcal{S}^{\text{step-}r}$ with $\pi^r(\tilde{F})=F$.
\end{definition}

Let us finish this Subsection with the following Lemma, which is an immediate consequence of Lemma \ref{InvarianceCoordinatesBrackets}. It shows that the differential relation $ \mathcal{S}^{\text{step-}r}$ does not depend on the choice of coordinates.

\begin{lemma}\label{DiffInvariance}
The differential relation $\mathcal{S}^{\step -r}$ is Diff(M) invariant.
\end{lemma}

\subsection{Existence of formal distributions.}\label{FormalExistence}

We aim to make this exposition as self-contained as possible, so let us review some elementary notions from Lie group theory that will be relevant in the upcoming discussion.

\begin{definition}\label{StratifiedLieAlgebra}
A Lie algebra $\mathfrak{g}$ is called \textbf{graded} if it is equipped with a grading compatible with the Lie bracket. It has a decomposition in vector spaces as follows:
\[\mathfrak{g}=\oplus_{i=1}^r \mathfrak{g}_i, \quad \text{where}\quad [\mathfrak{g}_i,\mathfrak{g}_j]\subset\mathfrak{g}_{i+j}.\]

If, in addition, $\mathfrak{g}_1$ generates $\mathfrak{g}$ as an algebra, then we say that $\mathfrak{g}$ is a \textbf{stratified algebra}; i.e. we have that
\[
\mathfrak{g}=\oplus_{i=1}^r \mathfrak{g}_i, \quad \text{where}\quad [\mathfrak{g}_1,\mathfrak{g}_i]=\mathfrak{g}_{i+1}.
\] 
The integer $r$ is called the \textbf{step} of the algebra $\mathfrak{g}$.
\end{definition}

\begin{definition}\label{StratifiedLieGroup}
Similarly, we say that a Lie group $G$ is a \textbf{stratified group} of step-$r$ if it is simply connected with a step-$r$ stratified Lie algebra $\mathfrak{g}=\oplus_{i=1}^r\mathfrak{g}_i$. We denote by $e\in G$ its identity element and we identify $T_eG$ with $\mathfrak{g}$.
\end{definition} 

\begin{remark}\label{nilpotency}
Stratified Lie algebras are nilpotent and, thus, so are stratified Lie groups.
\end{remark}

\begin{remark}\label{DistributionLieGroup} Note that, on stratified groups, the degree-one layer $(\mathfrak{g}_1)_e\subset T_eG$ defines, by left translation, a left-invariant distribution $\SD\subset TG$. Also, by definition, $\mathfrak{g}_1$ generates the whole Lie algebra.
\end{remark}

\begin{remark}\label{BCHformula}

The Baker-Campbell-Hausdorff formula is a well known result in the Theory of Lie groups (see, for instance, \cite[1.3]{R}). Consider a Lie algebra $\mathfrak{g}$ with corresponding connected Lie group $G$ and exponential map $\exp:\mathfrak{g}\to G$. Given two elements $X, Y\in\mathfrak{g}$, this formula is a formal series (not necessarily convergent):
\begin{equation}
Z(X,Y)=X + Y + \frac{1}{2} [X,Y] + \frac{1}{12} [X, [X,Y]] - \frac{1}{12} [Y, [X,Y]]-\frac{1}{24}[Y, [X, [X,Y]]]+\cdots
\end{equation}

such that, if it converges, then it produces the element $Z\in\mathfrak{g}$ solving the equation $\exp(X)\exp(Y)=\exp(Z)$.

Note that in the case of nilpotent Lie algebras and groups, the  Baker-Campbell-Haussdorff formula only contains a finite number of non-zero terms and thus becomes a closed finite expression. This readily allows to identify $\mathfrak{g}$ with $G$, via this formula, and $\mathfrak{g}$ is thus endowed with a Lie group structure by inheriting the group structure of $G$.
\end{remark}

The following notion is the analog of the Nilpotentisation (recall Definition \ref{Nilpotentisation}) in the formal setting.

\begin{definition}\label{MaxGrNilp}

Let $M$ be a smooth $n$-manifold. Consider a flag of subbundles of $TM$
\begin{equation*}\label{flag}
\SD_1\subset\cdots\subset\SD_{r-1}\subset\SD_r\subseteq TM.
\end{equation*}

together with a fibrewise stratified Lie algebra structure of step-$r$ (for some fibrewise Lie bracket) on
\begin{equation*}\label{GradedVectorBundle}
\SD_1 \oplus \SD_2/\SD_1 \oplus \cdots \oplus \SD_i/\SD_{i-1} \oplus \cdots \oplus \SD_r/\SD_{r-1}. 
\end{equation*} 

Such a fibrewise Lie algebra structure is called a \textbf{formal nilpotentisation} of step-$r$. If, additionally, each dimension $\dim\left(\SD_{i}/\SD_{i-1})\right)$ coincides with the $i$-th entry of a maximal growth vector for ambient dimension $n$ (in particular implying $\SD_r=TM$), we then say that the formal nilpotentisation is of \textbf{maximal growth}.\end{definition}

The next result and its proof are due to Álvaro del Pino. It states that the existence of a formal nilpotentisation of maximal growth on a smooth manifold implies the existence of a formal distribution of maximal growth.

\begin{lemma}\label{ExistenceLieAlgebraImpliesFormalDistribution}
If $M$ admits a step-$r$ formal nilpotentisation of maximal growth then it admits a formal distribution of $\step$-$r$.
\end{lemma}

\begin{proof}
Assume, by hypothesis, that we have a flag of subbundles
\begin{equation}\label{FlagNilp} \SD_1\subset\cdots\subset\SD_r=TM \end{equation}
whose corresponding graded vector space
 \begin{equation}\label{GradedVectorBundle2}
\SD_1 \oplus \SD_2/\SD_1 \oplus \cdots \oplus \SD_i/\SD_{i-1} \oplus \cdots \oplus \SD_r/\SD_{r-1}. 
\end{equation}
identifies fibrewise with a stratified Lie algebra $\mathfrak{g}=\oplus_{i=1}^r \mathfrak{g}_i$.

Take a point $p\in M$ and $T_pM\simeq \mathfrak{g}$. The Baker-Campbell-Haussdorff formula (Remark \ref{BCHformula}) identifies $\mathfrak{g}$ with the corresponding nilpotent Lie group $G$ which is stratified of step-$r$ (because $\mathfrak{g}$ was). We can thus consider the left-invariant distribution $\SD\subset TG$ associated to the first layer of the grading $\mathfrak{g}_1$ (recall Remark \ref{DistributionLieGroup}). By construction, the nilpotentisation associated to $\SD$ (recall Subsection \ref{sssec:nilpotentisation}) over any point of $G$ is isomorphic to the given one (\ref{GradedVectorBundle2}).

Fix a metric $g$ on $M$ and use the exponential map to identify $T_pM\simeq\mathfrak{g}$ with a small open neighborhood $\Op(p)\subset M$. We thus identify $\Op(p)\subset M$ with the Lie group $G$ and $p$ with $e\in G$. The $r-1$-jet  $j^{r-1}(\SD)$ defines a formal distribution over $\Op(p)$ with associated formal Lie flag (\ref{FlagNilp}). 

This construction clearly works with parameters. Therefore we can choose a fine enough covering of $M$ and carry out this process  over every open set in the covering simultaneuosly. We thus end up constructing a global formal distribution on $M$ with formal Lie flag (\ref{FlagNilp}).
\end{proof}

By means of the previous Lemma we can finally prove the main result in this Subsection.

\begin{proposition}[Parallelizability implies existence of a formal structure]\label{ExistenceFormal}
Let $M$ be an $n$-dimensional parallelizable manifold and fix $1<k<n$. Then $M$ admits a $k$-rank formal bracket-generating distribution of maximal growth.
\end{proposition}

\begin{proof}

By Lemma \ref{ExistenceLieAlgebraImpliesFormalDistribution} it suffices to show that $M$ admits a formal nilpotentisation of maximal growth. Denote by $(\mathfrak{n}_i)_{i=1}^r$ a maximal growth vector with $\mathfrak{n}_1=k, \mathfrak{n}_r=n$.

Choose any stratified Lie algebra $\mathfrak{g}=\oplus^r_{i=1}\mathfrak{g}_i$ so that $\dim(\mathfrak{g}_i)=\mathfrak{n}_i$. By the parallelizability of $M$, trivialise $TM$ and thus consider a fibrewise identification of $TM$ with $\mathfrak{g}$; i.e. we have fibrewise linear isomorphisms $i_p:\mathfrak{g} \to T_pM$, varying smoothly with $p\in M$. We can just translate the stratification of $\mathfrak{g}$ onto $TM$. In other words,  we can consider the pushforward $(\SD_i)_p:=(i_p)_{*}(\mathfrak{g}_i)$ of each $\mathfrak{g}_i$ and readily translate the stratified Lie Algebra structure of $\mathfrak{g}$ onto $TM$. 

We thus produce a flag of subbundles $\SD_1\subset\cdots\subset\SD_{r-1}\subset TM$ whose fibrewise associated graded vector bundle \begin{equation*}
\SD_1 \oplus \SD_2/\SD_1 \oplus \cdots \oplus \SD_i/\SD_{i-1} \oplus \cdots \oplus \SD_{r}/\SD_{r-1}
\end{equation*}
inherits the stratified Lie algebra structure of $\mathfrak{g}$ where each $\SD_i$ identifies with $\mathfrak{g}_i$. Therefore $M$ admits a formal nilpotentisation of maximal growth and, by Lemma \ref{ExistenceLieAlgebraImpliesFormalDistribution}, a formal distribution of maximal growth.
\end{proof}

\section{$h$-principle for distributions of maximal growth}\label{hprincipleMG}

The goal of this last Section is to prove the following Theorem, which will imply the main Theorem \ref{mainthm} in this article.

\begin{theorem}\label{AmplenessR}
Let $M$ be a smooth manifold of dimension $n$. The differential relation $\mathcal{R}^{\text{step-}r}\subset J^{r-1}(\Gr_k (TM))$ is ample for $k> 2$.
\end{theorem}

 In Subsection \ref{Red1} we present a new criterion to check ampleness of a Differential relation. This will constitute the first reduction step since it will allow us to reduce the problem of checking ampleness of  $\mathcal{R}^{\text{step-}r}$ to checking ampleness of  $\mathcal{S}^{\text{step-}r}$. On a second reduction step (Subsection \ref{Red2}) we will show that we can reduce the study of ampleness to checking ampleness just along some particular principal directions called \textit{non-normal}. In Subsection \ref{AdaptedFrames1} and \ref{AdaptedFrames2} we introduce and ellaborate on some particular models where checking ampleness becomes somehow tractable. Finally we carry out the argument for the proof of ampleness of $\mathcal{S}^{\text{step-}r}$ in Section \ref{AmplenessS}.

\subsection{A fibered criterion for ampleness. First reduction.}\label{Red1}

\subsubsection{Fibered criterion for ampleness}
There are several criteria to check ampleness of differential relations in certain contexts: thinness of its complement (Lemma \ref{thinSubsets}, see also \cite[18.4.2]{EM}), ampleness criterion for $A$-directed immersions (\cite[19.1.1]{EM}), elliptic operators $\mathcal{D}$ having $\rank\geq 2$ for relations defined by linear independance of systems of $\mathcal{D}-$sections (\cite[pp. 181-182]{Gro86}, \cite[20.4.1]{EM}), etc. Once any of these conditions is checked, one can directly invoke the Theorem of Convex integration (Theorem \ref{thm:convexIntegration1}) and the complete $h$-principle follows.

The goal of this subsection is to present a criterion of such type (that to our knowledge has not been explained elsewhere) that shows that ampleness of certain differential relations reduces, in a precise sense, to ampleness of some other differential relations fibering over them. 
As we will show in the next Section via a concrete example, checking ampleness of these auxilliary relations can in some cases be easier, thus showing that the criterion we present is not void.

The philosophical idea behind this criterion is as follows: in order to prove ampleness of a differential relation in $J^r(X)$, one may find easier to work with an auxiliary space $Y$ (which may play the role of being the space of frames/coordinates/additional structure of the space $X$). In particular, this allows to check (in a precise sense) ampleness of a certain Diff(M) invariant differential relation by just checking it for certain choice of local coordinates.

Let us fix the general framework. Consider $p:Y\to M$, $\rho:X\to M$ two smooth fiber bundles over the same base manifold and $\pi: Y\to X$ a surjective and submersive bundle map (that can often be thought as a projection or a quotient map). We thus have the induced map $\pi^r:J^r(Y)\to J^r(X)$. Consider an open differential relation $\mathcal{S}\subset J^r(Y)$ and and open differential relation $\mathcal{R}\subset J^r(X)$ so that the former is the pullback of the latter; i.e. ${\mathcal{S}:=(\pi^r})^*(\mathcal{R})$. We then have that $\pi^r\left(\mathcal{S}\right)=\mathcal{R}$ where fibers are mapped to fibers. This is summed up by the commutative diagram in Figure \ref{diagram}.

\begin{figure}[h]
\begin{center}
\begin{tikzpicture}[x=0.75pt,y=0.75pt,yscale=-1,xscale=1]

\draw    (84.13,132.83) -- (110.41,159.96) ;
\draw [shift={(111.8,161.4)}, rotate = 225.92] [color={rgb, 255:red, 0; green, 0; blue, 0 }  ][line width=0.75]    (10.93,-4.9) .. controls (6.95,-2.3) and (3.31,-0.67) .. (0,0) .. controls (3.31,0.67) and (6.95,2.3) .. (10.93,4.9)   ;
\draw    (172.27,72.43) -- (172.58,104.6) ;
\draw [shift={(172.6,106.6)}, rotate = 269.44] [color={rgb, 255:red, 0; green, 0; blue, 0 }  ][line width=0.75]    (10.93,-4.9) .. controls (6.95,-2.3) and (3.31,-0.67) .. (0,0) .. controls (3.31,0.67) and (6.95,2.3) .. (10.93,4.9)   ;
\draw    (161.93,132.6) -- (134.03,160) ;
\draw [shift={(132.6,161.4)}, rotate = 315.53] [color={rgb, 255:red, 0; green, 0; blue, 0 }  ][line width=0.75]    (10.93,-4.9) .. controls (6.95,-2.3) and (3.31,-0.67) .. (0,0) .. controls (3.31,0.67) and (6.95,2.3) .. (10.93,4.9)   ;
\draw    (72.27,72.43) -- (72.58,104.6) ;
\draw [shift={(72.6,106.6)}, rotate = 269.44] [color={rgb, 255:red, 0; green, 0; blue, 0 }  ][line width=0.75]    (10.93,-4.9) .. controls (6.95,-2.3) and (3.31,-0.67) .. (0,0) .. controls (3.31,0.67) and (6.95,2.3) .. (10.93,4.9)   ;
\draw    (99.4,122.2) -- (145,122.58) ;
\draw [shift={(147,122.6)}, rotate = 180.48] [color={rgb, 255:red, 0; green, 0; blue, 0 }  ][line width=0.75]    (10.93,-4.9) .. controls (6.95,-2.3) and (3.31,-0.67) .. (0,0) .. controls (3.31,0.67) and (6.95,2.3) .. (10.93,4.9)   ;
\draw    (99.8,54.2) -- (145.4,54.2) ;
\draw [shift={(147.4,54.2)}, rotate = 180.48] [color={rgb, 255:red, 0; green, 0; blue, 0 }  ][line width=0.75]    (10.93,-4.9) .. controls (6.95,-2.3) and (3.31,-0.67) .. (0,0) .. controls (3.31,0.67) and (6.95,2.3) .. (10.93,4.9)   ;

\draw (67,115.4) node [anchor=north west][inner sep=0.75pt]    {$Y$};
\draw (164.67,113.6) node [anchor=north west][inner sep=0.75pt]    {$X$};
\draw (113.8,164.8) node [anchor=north west][inner sep=0.75pt]    {$M$};
\draw (56.4,49) node [anchor=north west][inner sep=0.75pt]    {$J^{r}( Y)$};
\draw (152.27,49) node [anchor=north west][inner sep=0.75pt]    {$J^{r}( X)$};
\draw (158.07,143.53) node [anchor=north west][inner sep=0.75pt]  [font=\small]  {$\rho $};
\draw (77.07,145.2) node [anchor=north west][inner sep=0.75pt]  [font=\small]  {$p$};
\draw (115.27,109) node [anchor=north west][inner sep=0.75pt]  [font=\small]  {$\pi $};
\draw (114.87,38) node [anchor=north west][inner sep=0.75pt]  [font=\small]  {$\pi ^{r}$};
\draw (27.4,49) node [anchor=north west][inner sep=0.75pt]    {$\mathcal{S} \subset $};
\draw (189.4,49) node [anchor=north west][inner sep=0.75pt]    {$\supset \mathcal{R}$};

\end{tikzpicture}

\end{center}
\caption{Submersive bundle map $\pi:Y\to X$ between smooth fiber bundles $Y\to M$ and $X\to M$ so that the associated $r$-jet extension $\pi^r$ maps the differential relation $\mathcal{S}$ into $\mathcal{R}$. The map $\pi$ can often be thought as a quotient map.}\label{diagram}
\end{figure}

\begin{lemma}\label{ExistenceLift}
Every $F\in J^r(X)$ has a lift $\tilde{F}\in J^r(Y)$ so that $\pi^r(\tilde{F})=F$.
\end{lemma}
\begin{proof}
This follows by the surjectivity and the submersive condition of $\pi$. Indeed, choose a local model where $\pi$ is regarded, locally, as a projection map and thus the statement readily follows.
\end{proof}

The following Lemma shows that $\pi^r: J^r(Y)\to J^r(X)$ is locally an affine bundle that maps principal subspaces to principal subspaces.

\begin{lemma}\label{AffineFibration}
For any $q\in Y$ there exists an open neighborhood $\mathcal{U}\subset Y$ of $q$ such that:
\begin{itemize}
\item[i)] $\pi^r: J^r(\mathcal{U})\to J^r(\mathcal{V})$ is an affine bundle, where $\mathcal{V}:=\pi(\mathcal{U})$.
\item[ii)] Principal subspaces of the fibration $J^r(\mathcal{U})\to p(\mathcal{U})$ are preimages of the principal subspaces of $J^r(\mathcal{V})\to\rho(\mathcal{V})$: 
\[\Pr^r_{\tau,F}=(\pi^r)^{-1}\left(\Pr^r_{\tau,\pi^r(F)}\right).\]

\end{itemize}
\end{lemma}

\begin{proof}
For any point $q\in Y$ and $x:=\pi(q)\in X$, choose local charts $q\in\mathcal{U}\subset Y$, $x\in\mathcal{V}\subset X$, so that both fibrations are trivial; i.e. $\mathcal{U}\simeq\R^n\times\R^m$, $\mathcal{V}\simeq\R^n\times\R^{\tilde{m}}$ where $n=\dim(M)$ and the fiber $\R^m$ factor submers onto the fiber $\R^{\tilde{m}}$ factor via $\pi$.

We can thus extend these coordinates (recall Section \ref{ssec:jetsCoord}) to the level of $r$-jets as follows:

\begin{equation}\label{Fibers1}
J^r(Y)\supset J^r(\mathcal{U})\simeq J^r(\R^n\times \R^m)\simeq \R^n\times \R^m\times \Hom(\R^n, \R^m)\times\Sym^2(\R^n, \R^m)\times\cdots\times\Sym^r(\R^n, \R^m),
\end{equation}
\begin{equation}\label{Fibers2} 
J^r(X)\supset J^r(\mathcal{V})\simeq J^r(\R^n\times \R^{\tilde{m}})\simeq \R^n\times \R^{\tilde{m}}\times \Hom(\R^n, \R^{\tilde{m}})\times\Sym^2(\R^n, \R^{\tilde{m}})\times\cdots\times\Sym^r(\R^n, \R^{\tilde{m}}).
\end{equation}

 Therefore the induced map $\pi^r: J^r(\mathcal{U}) \buildrel{\pi^r}\over{\to} J^r(\mathcal{V}) $ maps each of the factors in the fibers accordingly; i.e. each factor in (\ref{Fibers1})  is mapped to its homologous factor in (\ref{Fibers2}).
We have that $\pi$ (which acts on the $j^0$-part of the fibers) is a trivial affine bundle (and note that $\pi$ is just the truncated part of $\pi^r$ acting on the first two factors $\R^n\times\R^m$ of the product; $\pi:\R^n\times\R^m\to\R^n\times\R^{\tilde{m}}$). The rest of the factors of the fibers are affine spaces which are clearly mapped affinely by $\pi^r$, thus yielding $i)$.
Once we have this local description of jet spaces as a product, $ii)$ is now just a direct consequence of $i)$.
\end{proof}

\begin{remark}\label{CriterionSmallerJets}
Note that if we restrict the whole construction to lower jet spaces $J^i(Y) \subset J^r(Y)$, $J^i(X) \subset J^r(Y)$, $i<r$, the same result follows for $i$-jets. Indeed, we have the projected relations $\pi^r_i(\mathcal{S})$, $\pi^r_i(\SR)$ and the map $\pi^i$ is just the truncated part of $\pi^r$ acting on $i$-jets. Therefore, the same proof applies verbatim in this context.
\end{remark}

As a consequence of Lemma \ref{AffineFibration}, we can establish the following Lemmas.

\begin{lemma}\label{FC1}
Consider $F\in J^r(X)$ and a lift $\tilde{F}\in J^r(Y)$; i.e. so that $\pi^r(\tilde{F})=F$. Take a principal codirection $\tau$. Then, for any $i\leq r$, ampleness of $\mathcal{S}^i_{\tau, \tilde{F}}\subset \Pr^i_{\tau,\tilde{F}}$  is equivalent to ampleness of $\mathcal{R}^i_{\tau, F}\subset \Pr^i_{\tau, F}$.
\end{lemma}

\begin{proof}
By Lemma \ref{AffineFibration}, $\pi^r$ becomes an affine fibration when restricted to small enough open sets of $J ^r(Y)$ that also maps principal subspaces upstairs $\Pr^r_{\lambda,\tilde{F}}$ to principal subspaces downstairs $\Pr^r_{\lambda,F}$ (and also $\mathcal{S}^r_{\tau, \tilde{F}}$ upstairs to $\mathcal{R}^r_{\tau, F}$ downstairs). We then have that convex combinations within $\mathcal{S}^r_{\tau, \tilde{F}}$ upstairs averaging\footnote{We say that a convex combination $\sum_i \lambda_i p_i$ (where $\sum_i\lambda_i=1$) averages $F$ if  $\sum_i \lambda_i p_i=F$.} $\tilde{F}$ translate, via $\pi^r$, to the corresponding convex combinations in $\mathcal{R}^r_{\tau, F}$ downstairs averaging $F$. 

Conversely, we want to show that a convex combination $\sum_i \lambda_i p_i$ in $\mathcal{R}^r_{\tau, F}$ averaging $F$ (where $p_i\in \mathcal{R}^r_{\tau, F}$) can be lifted to a convex combination in $\mathcal{S}^r_{\tau, \tilde{F}}$ averaging $\tilde{F}$. By Lemma \ref{AffineFibration}, $\pi^r$ locally becomes, upon choosing small enough neighborhoods, a trivial affine fibration $\pi^r: J^r(\mathcal{V})\times W\to J^r(\mathcal{V})$, where $W$ denotes the affine fiber. Let  $J^r(\mathcal{V})\times\lbrace w\rbrace$ be the leaf containing $\tilde{F}$. It is clear then that the convex combination $\sum_i \lambda_i \tilde{p}_i$ does the job, where $\tilde{p}_i:= \lbrace p_i\rbrace\times\lbrace w\rbrace$.

Arguing in the same fashion for $\SR^i_{\tau, F}$ and $\Pr^i_{\lambda, F}$  (recall the notation from Remark \ref{NotationPr}), for $i\leq r$, the claim follows as well (Remark \ref{CriterionSmallerJets}).
\end{proof}

Let us state an obvious consequence of Lemma \ref{FC1}.

\begin{lemma}\label{CondicionCriterio}
Let $F, G\in J^r(Y)$ so that $\pi^r(F)=\pi^r(G)$. Then, for $1\leq i\leq r$, ampleness of $\mathcal{S}^i_{\tau, F}\subset \Pr^i_{\tau, F}$ is equivalent to ampleness of $\mathcal{S}^i_{\tau, G}\subset \Pr^i_{\tau, G}$ for any principal codirection $\tau$.
\end{lemma}

We now state the ampleness-criterion.

\begin{proposition}[\textbf{Fibered ampleness-criterion}]\label{AmplenessCriterion}
Ampleness of $\mathcal{S}\subset J^r(Y)$ is equivalent to ampleness of $\mathcal{R}\subset J^r(X)$. 
\end{proposition}

\begin{proof}

If $\mathcal{S}$ is ample, then for any codirection $\tau$, any $i\leq r$ and any $F\in J^r(X)$ there exists a lift $\tilde{F}\in J^r(Y)$ (Lemma \ref{ExistenceLift}) so that $\pi^r(\tilde{F})=F$ and $\mathcal{S}^i_{\tau, \tilde{F}}\subset \Pr^i_{\tau,\tilde{F}}$ is ample. Therefore, it follows by Lemma \ref{FC1} that, for any codirection $\tau$, any $i\leq r$ and any $F\in J^r(X)$,  $\mathcal{R}^i_{\tau, {F}}\subset \Pr^i_{\tau,{F}}$ is ample, thus yielding ampleness of $\mathcal{R}$.

Conversely, if $\SR$ is ample, then for any codirection $\tau$, any $i\leq r$ and any ${G}\in J^r(Y)$, we have that $\mathcal{R}^i_{\tau, \pi^r(G)}\subset \Pr^i_{\tau,\pi^r(G)}$ is ample. Then, by Lemma \ref{FC1}, we have that for any codirection $\tau$, any $i\leq r$ and any  $G\in J^r(Y)$ it follows that $\mathcal{R}^i_{\tau, G}\subset \Pr^i_{\tau, G}$ is ample. This yields ampleness of $\mathcal{S}$.
\end{proof}

We will apply the previous criterion for the case of a differential relation in the context of distributions on manifolds.

\subsubsection{Particularising to distributions.}

We will apply the Fibered Criterion in order to reduce ampleness of $\mathcal{R}^{\step -r}$ to some easier to check condition on the bundle of $k$-frames.
Recall (Section \ref{FBD}) that we have the surjective and submersive bundle map $\pi:\overline{\bigoplus_k TM}\to\Gr_k(TM)$ which maps each $k$-frame to the $k$-plane it spans and thus induces a map at the level of $r-$jets:

\begin{figure}[h]
\begin{center}
\begin{tikzpicture}[x=0.75pt,y=0.75pt,yscale=-1,xscale=1]

\draw    (226.27,72.43) -- (226.27,104.6) ; 
\draw [shift={(226.27,106.6)}, rotate = 269.44] [color={rgb, 255:red, 0; green, 0; blue, 0 }  ][line width=0.75]    (10.93,-4.9) .. controls (6.95,-2.3) and (3.31,-0.67) .. (0,0) .. controls (3.31,0.67) and (6.95,2.3) .. (10.93,4.9)   ;
\draw    (87.27,72.43) -- (87.27,104.6) ; 
\draw [shift={(87.27,106.6)}, rotate = 269.44] [color={rgb, 255:red, 0; green, 0; blue, 0 }  ][line width=0.75]    (10.93,-4.9) .. controls (6.95,-2.3) and (3.31,-0.67) .. (0,0) .. controls (3.31,0.67) and (6.95,2.3) .. (10.93,4.9)   ;
\draw    (135.4,122.2) -- (179.4,122.2) ; 
\draw [shift={(179.4,122.2)}, rotate = 180.48] [color={rgb, 255:red, 0; green, 0; blue, 0 }  ][line width=0.75]    (10.93,-4.9) .. controls (6.95,-2.3) and (3.31,-0.67) .. (0,0) .. controls (3.31,0.67) and (6.95,2.3) .. (10.93,4.9)   ;
\draw    (135.4,57.2) -- (179.4,57.2) ;
\draw [shift={(177.4,57.2)}, rotate = 180.48] [color={rgb, 255:red, 0; green, 0; blue, 0 }  ][line width=0.75]    (10.93,-4.9) .. controls (6.95,-2.3) and (3.31,-0.67) .. (0,0) .. controls (3.31,0.67) and (6.95,2.3) .. (10.93,4.9)   ;

\draw (65,115.4) node [anchor=north west][inner sep=0.75pt]    {$\overline{\bigoplus_k TM}$};
\draw (198.67,115.6) node [anchor=north west][inner sep=0.75pt]    {$\Gr_k(TM)$};
\draw (56.4,49) node [anchor=north west][inner sep=0.75pt]    {$J^{r}( \overline{\bigoplus_k TM})$};
\draw (186.27,50) node [anchor=north west][inner sep=0.75pt]    {$J^{r}( \Gr_k(TM))$};
\draw (148.87,109) node [anchor=north west][inner sep=0.75pt]  [font=\small]  {$\pi $};
\draw (148.87,41) node [anchor=north west][inner sep=0.75pt]  [font=\small]  {$\pi ^{r}$};
\draw (-3,50) node [anchor=north west][inner sep=0.75pt]    {$\mathcal{S}^{\step -r} \subset $};
\draw (270.4,50) node [anchor=north west][inner sep=0.75pt]    {$\supset \mathcal{R}^{\step -r}$};

\end{tikzpicture}
\end{center}
\end{figure}

We make the choices $Y= \overline{\bigoplus_k TM}$, $X=\Gr_k(TM)$ and the corresponding differential relations $\mathcal{S}^{\step -r}\subset J^{r}( \overline{\bigoplus_k TM})$ and $\mathcal{R}^{\step -r}\subset J^{r}( \Gr_k(TM))$. 

By Proposition \ref{AmplenessCriterion}, ampleness of $\mathcal{R}^{\step -r}$ is equivalent to ampleness of $\mathcal{S}^{\step -r}$. We will prove that $\mathcal{S}^{\step -r}$ is ample. This will, thus, yield ampleness of $\mathcal{R}^{\step -r}$.

\subsection{ Second reduction: restricting to non-normal principal directions.}\label{Red2}

The goal of this step is to check that ampleness of $\mathcal{S}^{\step -r}$ with respect to certain directions called ``normal'' follows trivially. This will imply that, after this step, we can reduce the study of ampleness of $\mathcal{S}^{\step -r}$ to non-normal principal directions. This will be encoded by Lemma \ref{PerpendicularEffect2}.

Fix a formal solution $F=(F_i)_{i=1}^k\in J^{r-1}( \overline{\bigoplus_k TM})$ and write $\SD_F=\langle j^0(F_1),\cdots, j^0(F_k)\rangle$ (we often drop the subindex and just write $\SD$ if it is clear from the context). 

\begin{remark}\label{RemarkPartialt}
Since convex integration works locally; i.e. it is implemented chart by chart, we can assume for the rest of the Section that $M=\R^n$. 
Up to changes of coordinates, when checking ampleness with respect to a certain direction $X_i(p)\in T_p\R^n$ over a point $p\in\R^n$, we can assume without loss of generality (recall Lemma \ref{DiffInvariance}) that this direction coincides with the first coordinate direction; i.e. $X_i(p)=\partial_t$.  

Similarly, we will consider the splitting $\R^n=\R[t]\times\R^{n-1}$ which induces the splitting of $F$ in Equation (\ref{FormalSplitting}). 
So, if it is not explicitly stated otherwise, whenever we write $\partial_t$ we are referring to $\partial_1$. Nonetheless, we will consistently use the letter $t$ for didactical reasons along the exposition.
\end{remark}

Whenever we talk about metric properties such as orthogonality, we will be referring to the Euclidean metric of $\R^n$. For the ease of notation we will denote $\mathcal{S}$ for $\mathcal{S}^{\step -r}$ along the rest of the Subsection.
 
\begin{definition}
A locally defined non-vanishing direction/vector field $X_i:\mathcal{U}\subset \R^n\to T\R^n$ is said to be \textbf{normal} at a point $p\in \R^n$ with respect to $F$ if it is normal to the distribution $\SD_F$  with respect to the Euclidean metric of $\R^n$ over $p$; i.e. $X_i(p)\perp\SD_F(p)$. Otherwise, we say that it is \textbf{non-normal}.
\end{definition}

\subsubsection{Principal subspaces} Recall that the principal subspaces  $\Pr^s_{\partial_t,F}$    correspond to the space of all $s$-jets that differ with $F$ only on the formal pure order-$s$  partial derivatives $j^s_t(F)=\left(j^s_t(F_i)\right)_{i=1}^k$ with respect to $\partial_t$ (where the usual identification of $F$ with $\pi^r_s(F)$ has been made, recall Remark \ref{NotationPr}).

Then, a principal subspace $\Pr^s_{\partial_t, F}$ can be expressed as:

\begin{equation}\label{PrincipalSubspaceDescription}
\Pr^s_{\partial_t, F}=\biggl\{ 
\left( (j^{\bot(\partial_t, s)}(F_i), 0) + (0, w_i)\right)_{i=1}^k: w_i\in\R^n
\biggr\},
\end{equation}

where the free $w_i\in \R^n$ parametrize all possible order-$s$ pure formal derivatives of $F_i$ with respect to $\partial_t$. 

\subsubsection{Normal directions and trivial ampleness}

Before proving the aforementioned Lemma \ref{PerpendicularEffect2}, we first state a more general result (Lemma \ref{PerpendicularEffect1}) which will not only imply Lemma  \ref{PerpendicularEffect2} but which will also be useful later on in this article.

Consider the first coordinate direction $\partial_t\in T_p\R^n$ over $p\in \R^n$ and a bracket expression $A_{(a_1,\cdots, a_\mathfrak{r})}(\cdot)$ of length $\mathfrak{r}\leq r$. Next Lemma states that if the expression $A_{(a_1,\cdots, a_\mathfrak{r})}(F)$  involves two components of $F$ with $j^0$-components orthogonal to $\partial_t$, then the whole expression does not depend on $j^{\mathfrak{r}-1}_t(F)$.

\begin{lemma}\label{PerpendicularEffect1} Consider $F=(F_i)_{i=1}^k\in J^{r-1}( \overline{\bigoplus_k T\R^n})$ over $p\in\R^n$. If a formal length-$\mathfrak{r}$ bracket $A_{(a_1,\cdots, a_{\mathfrak{r}})}(F)$ involves at least two components $F_{a_i}, F_{a_j}$ $(i,j\leq\mathfrak{r})$ from $F$ satisfying 

\begin{equation}\label{OthogonalityConditionComponents}
 j^0(F_{a_i}), j^0(F_{a_j})\perp\partial_t,
\end{equation}

 then it does not depend on $j^{\mathfrak{r}-1}_t(F)$. More precisely, there exists a map $h$ such that for all $\tilde{F}\in J^{r-1}( \overline{\bigoplus_k T\R^n})$ whose $j^0(F_{a_i}), j^0(F_{a_j})$ components satisfy Condition (\ref{OthogonalityConditionComponents}), the expression $A_{(a_1,\cdots, a_{\mathfrak{r}})}$ factors through $h$ as follows:

\begin{equation}
A_{(a_1,\cdots, a_{\mathfrak{r}})}(\tilde{F})={h}\circ  j^{\bot(\partial_t,\mathfrak{r-1})}(\tilde{F}). 
\end{equation}
\end{lemma}

\begin{remark}
Note that, potentially, $F_{a_i}$ could denote the same component in $F$ as $F_{a_j}$. That would be the case if $a_i=a_j$ and, thus, the corresponding component appeared twice in the entries of $A_{(a_1,\cdots, a_r)}(F)$.
\end{remark}

\begin{proof}[Proof of Lemma \ref{PerpendicularEffect1}]
Since $\partial_t\perp j^0(F_{a_i}), j^0(F_{a_j})$, then these two vectors do not have $\partial_t$-component when written in coordinates:
\[
j^0(F_{a_i})=\color{blue}0\color{black}\cdot\partial_t + u^{a_i}_2\cdot\partial_2 + \cdots + u^{a_i}_n\cdot\partial_n,
\]
\[
j^0(F_{a_j})=\color{blue}0\color{black}\cdot\partial_t + u^{a_j}_2\cdot\partial_2 + \cdots +  u^{a_j}_n\cdot\partial_n.
\]

Express the bracket $A_{(a_1,\cdots, a_r)}(F)$   in simplified form (Lemma \ref{ExistenceSimplifiedForm}).  The maximal number $m$ of iterated derivations of the form $D_t\circ\cdots\circ D_t(u_{\ell_i}^q)$ that could potentially appear in the expression is at most $\mathfrak{r}-2$ (one per each contribution of the $\partial_t$-component of each $j^0(F_{a_m})$, $a_m\neq {a_i},{a_j}$). This means that such expression can be expressed purely in terms of jet-coordinates $u_{(i, J)}$ with $J\in\mathcal{S}_{\mathfrak{r}-1}$ and $J\neq (t,\cdots, t)\in\NS_{\mathfrak{r}-1}$. 

We thus conclude (Lemma \ref{SplittingDt}) that $A_{(a_1,\cdots, a_{\mathfrak{r}})}(\tilde{F})$ does not depend on $j^{\mathfrak{r}-1}_t(F)$ or, equivalently,  there is a factorization through some map $h$ as follows: $A_{(a_1,\cdots, a_\mathfrak{r})}(\cdot)={h}\circ  j^{\bot(\partial_t,{\mathfrak{r}-1})}(\cdot)$ for such jets satisfying Condition (\ref{OthogonalityConditionComponents}).\end{proof}

The following Lemma shows that any change to the pure order-$\mathfrak{r}-1$ information with respect to a direction $\partial_t$ perpendicular to $\SD$ does not contribute at all to changing the length-$\mathfrak{r}$ brackets of elements in $\SD$ (and, so, neither the bracket-generating condition defined based on them).

\begin{lemma}\label{PerpendicularEffect2} For jets $F=(F_i)_{i=1}^k\in J^{r-1}( \overline{\bigoplus_k T\R^n})$ over $p\in\R^n$ for which $\partial_t\perp\SD_F$, length-$\mathfrak{r}$ brackets $A_{(a_1,\cdots, a_\mathfrak{r})}(F)$ do not depend on $j^{\mathfrak{r}-1}_t(F)$ . Equivalently, there exists a map $h$ so that for such jets the bracket factors as:
\begin{align*}
A_{(a_1,\cdots, a_\mathfrak{r})}(F)={h}\circ j^{\bot(\partial_t,{\mathfrak{r}-1})}(F)
\end{align*}

\end{lemma}

\begin{figure}[h] 
	\includegraphics[scale=0.475]{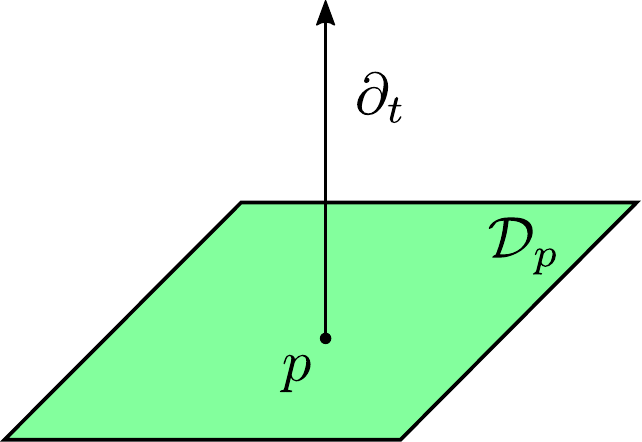}
	\centering
	\caption{When $\partial_t\perp\SD_F$ at a point $p\in \R^n$, any formal bracket of length $\mathfrak{r}$ involving components of $F$ does not depend on the pure $\mathfrak{r}-1$-order information in the direction of $\partial_t$ (Lemma \ref{PerpendicularEffect2}). Therefore, ampleness of $ \mathcal{S}^i_{\partial_t, F}$ within $ \Pr^i_{\partial_t,F}$ holds trivially for all $1\leq i \leq r-1$ in this case (Corollary \ref{TrivialAmpleness}). }\label{orthogonal}
\end{figure}
\begin{proof}
Since $\partial_t\perp\SD_p$, no $j^0(F_a)$ has $\partial_t$-component; i.e. for any $a\in\lbrace 1,\cdots, k\rbrace$ we can write:

\[
j^0(F_a)=\sum_{i=1}^n a_i\partial_i = \color{blue}0\color{black}\cdot\partial_t + a_2\cdot\partial_2 + \cdots +  a_n\cdot\partial_n.
\]

The result then readily follows by Lemma \ref{PerpendicularEffect1}.
\end{proof}

\begin{remark}
Note that we did not use the fact that $k>2$ in the previous Lemmas; i.e. they apply to the $\rank$-$2$ case as well.
\end{remark}

\begin{corollary}\label{TrivialAmpleness}
If $\partial_t\perp\SD_F$ then $ \mathcal{S}^i_{\partial_t, F}$ is trivially ample in $ \Pr^i_{\partial_t,F}$ for $1\leq i\leq r-1$.
\end{corollary}

\begin{proof}

If $F\in\mathcal{S}^{\step -r}$,  then $\dim\left(\langle\mathfrak{Br}^j\rangle\right)=\mathfrak{n}_j$ for all $1\leq j\leq r$, where $(\mathfrak{n}_j)_{j=1}^r$ is a maximal growth vector (by Definition \ref{formalBG}). Take any $i=1,\cdots, r-1$. Any other $\tilde{F}\in \Pr^i_{\partial_t, F}$ only differs from $F$ in the pure order-$i$ formal partial derivatives with respect to $\partial_t$ (Equation \ref{PrincipalSubspaceDescription}) and, therefore, by Lemma \ref{PerpendicularEffect2}, it has the same associated sets $\mathfrak{Br}^j$ for $1\leq j\leq r$ (thus satisfying the formally maximal growth condition as well). We conclude thus that any $\tilde{F}\in \Pr^i_{\partial_t, F}$ is a formal solution of  $\mathcal{S}^{\step -r}$, yielding trivial ampleness of $ \mathcal{S}^i_{\partial_t, F}$ within $ \Pr^i_{\partial_t,F}$ for $1\leq i\leq r-1$.\end{proof}

\subsection{Adapted frames with respect to non-normal directions}\label{AdaptedFrames1}

The goal of this Subsection is to introduce some local frames associated to a given distribution and a non-normal direction where ampleness of $\mathcal{S}^{\step -r}$ can be easily checked.

\begin{definition}

Let $\SD$ be a distribution on $\R^n$. We say that a frame $\mathfrak{Fr}=\lbrace X_1,\cdots, X_k\rbrace$ of $\SD$  is \textbf{ adapted to} the direction $\partial_t$ if the following conditions hold pointwise:

\begin{itemize}

\item[i)] $X_1$ is the orthogonal projection of $\partial_t$ onto the distribution $\SD$.
\item[ii)] For  $i\geq 2$, the vector fields $X_i$ are orthogonal to $X_1$; i.e. they lie in $\langle X_1\rangle^\perp\cap\SD$.\end{itemize}
\end{definition}

\begin{figure}[h] 
	\includegraphics[scale=0.34]{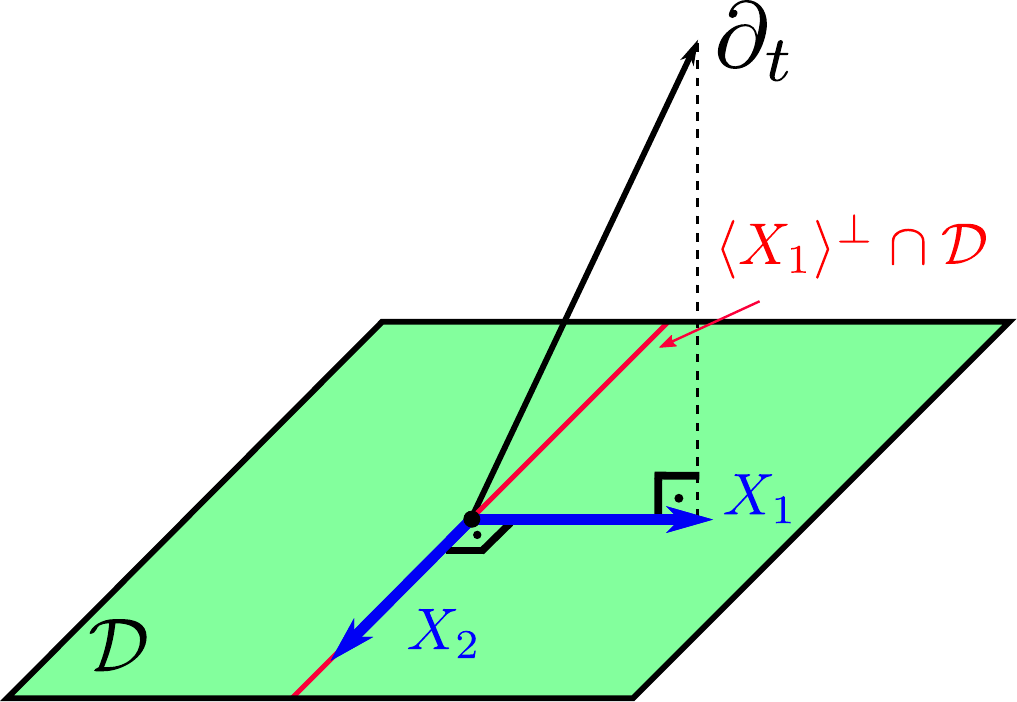}
	\centering
	\caption{Example of a  rank-$2$ distribution $\SD$ and a frame $\mathfrak{Fr}=\lbrace X_1, X_2\rbrace$ (in blue) adapted to the non-normal direction $\partial_t$: note that $X_1$ is the normal projection of $\partial_t$ onto $\SD$ by the Euclidean metric of $\R^3$ and $X_2$ lies in $\langle X_1\rangle^\perp\cap\SD$ (in red).  }\label{AdaptedFrame}
\end{figure}

\begin{remark}\label{orth}
Conditions $i)$ and $ii)$ imply, in particular, that  $X_i\perp\partial_t$ for $i\geq 2$. See Figure \ref{AdaptedFrame}.
\end{remark}

The definition above naturally extends to the context of $r-1$-jets.

\begin{definition}

Take $F=(F_i)_{i=1}^k\in J^{r-1}( \overline{\bigoplus_k T\R^n})$. We say that $F$ at $p\in\R^n$ is \textbf{adapted to }$\partial_t$ if the frame $\lbrace j^0(F_1),\cdots, j^0(F_k)\rbrace$ of $\SD:=\langle j^0(F_1),\cdots, j^0(F_k)\rangle$  is adapted to $\partial_t$ at $p$.
\end{definition}

Since we are working in coordinates where $\partial_t$ denotes the first coordinate direction of $\R^n$ (recall Remark \ref{RemarkPartialt}), the definition above implies the elements $j^0(F_j)$ to admit expressions of the form:
\begin{align}\label{AdaptedInCoordinates}
j^0(F_1) =\phantom{,} & 1\cdot \partial_t + \sum_{\substack{j=2}}^n u_1^j\partial_j, \quad u_j^1\in\R \\\label{AdaptedInCoordinates2}
j^0(F_i) =\phantom{,} & 0\cdot\partial_t + \sum_{\substack{j=2}}^n u^j_i\partial_j  \quad\text{ for }i\geq 2, \quad u_i^j\in\R.
\end{align}

\begin{lemma}[Existence of $\partial_t$-adapted frames]\label{ExistenceAF}
Consider a distribution $\SD$ and a point $p\in \R^n$. If the integral field $X_t\subset T\R^n$ of the coordinate direction $\partial_t$ is non-normal with respect to $\SD$, then there exists a local frame of $\SD$ which is adapted to $X_t$ in a small neighborhood $\Op(p)$.
\end{lemma}

\begin{proof}
Since non-normality is an open condition, there exists, on a small neighborhood $\Op(p)$ of $p$, a non-zero smooth projection $X_1$ of $X_t$ onto $\SD$. We can always locally complete $X_1$ to a frame $\mathfrak{Fr}=\lbrace X_1,\cdots, X_k\rbrace$ of $\SD$ adapted to $\partial_t$ by considering a frame of $\langle X_1\rangle^\perp$ inside $\SD$. This proves the local existence of adapted frames.
\end{proof}

Lemma \ref{ExistenceAF} readily implies, in terms of jets, the following lemma.

\begin{lemma}[Existence of $\partial_t$-adapted jets]\label{ExistenceAdapted}
Consider $F\in J^{r-1}(\Gr_k(T\R^n))$ over a point $p\in\R^n$ such that the direction $\partial_t\in T_p\R^n$ is not normal to $\SD_F= j^0(F)$. Then there always exists an element  $\tilde{F}=(\tilde{F}_i)_{i=1}^k\in J^{r-1}( \overline{\bigoplus_k T\R^n})$ projecting to $F$ by $\pi^r$ (recall Section \ref{FBD}) which is adapted to $\partial_t$.
\end{lemma}

\subsection{Principal subspaces and adapted frames}\label{AdaptedFrames2}

We will give in this Subsection a description of principal subspaces and their intersection with $\mathcal{S}^{\step -r}$  in terms of adapted frames. We will introduce some notation first.

Consider some jet $\tilde{F}\in J^{r-1}\left(\Gr_k(T\R^n)\right)$. Take now $F=(F_1,\cdots, F_k)\in J^{r-1}( \overline{\bigoplus_k T\R^n})$ adapted to $\partial_t$ and so that $\pi^{r-1}(F)=\tilde{F}$ (which exists by Lemma \ref{ExistenceAdapted}).

Consider $\mathfrak{Br}^1:=\lbrace j^0(F_1), j^0(F_2),\cdots, j^0(F_{k})\rbrace$ so that $\SD_F=\langle  j^0(F_1), j^0(F_2),\cdots, j^0(F_{k})\rangle$. 

Let us split the set of indices $\mathfrak{I}_i=\lbrace (\ell_1,\cdots, \ell_j) : j\leq i \text{ and } \ell_1,\cdots, \ell_j\in\lbrace 1,\cdots, k\rbrace\rbrace\cup\lbrace O\rbrace$ (recall Definition \ref{OrderedMI}) into 

\begin{equation}\label{SplittingSubindices} 
\mathfrak{I}_i=\mathfrak{I}_i^t\sqcup \tilde{\mathfrak{I}}_i,
\end{equation}

where  $\tilde{\mathfrak{I}}_i$ denotes the set of multi-indices of length less or equal than $i$ where the index $1$ does not appear $i-1$ times. More precisely, denote by $\Sigma_j$ the symmetric group of degree $j$. Then:

\begin{align*} 
&\tilde{\mathfrak{I}}_i:= \Big\{ (\ell_1,\cdots, \ell_j) : 0\leq j\leq i,\quad \forall\sigma\in\Sigma_{j}\quad \forall m\leq k,\quad \sigma  (\ell_1,\cdots, \ell_j) \neq(\underbrace{1,\cdots, 1}_{i-1}, m)   \Big\} \\
&\mathfrak{I}_i^t:= \Big\{ (\ell_1,\cdots, \ell_i) : \exists \sigma\in\Sigma_{i}\quad \exists m\leq k,\quad \sigma  (\ell_1,\cdots, \ell_i) = (\underbrace{1,\cdots, 1}_{i-1}, m)\Big\}.
\end{align*}
\begin{remark}
The case $j=0$ in the definition of the set $\tilde{\mathfrak{I}}_i$ corresponds to the empty subindex $O=()$ which also belongs to the set.
\end{remark}

Splitting (\ref{SplittingSubindices}) induces a splitting in $\mathfrak{Br}^{i}$ (recall Definition \ref{BrFormal} and Remark \ref{ommitF}) as $\mathfrak{Br}^{i}=  {\mathfrak{Br}}_{\perp, t}^{i} \sqcup \mathfrak{Br}^{i}_t$, where 
\begin{equation}\label{Br2types} {\mathfrak{Br}}_{\perp, t}^{i}=\lbrace A_J(F): J\in \tilde{\mathfrak{I}}_i\rbrace \quad \text{and}\quad \mathfrak{Br}^{i}_t=\lbrace A_J(F): J\in\mathfrak{I}_i^t\rbrace.\end{equation}

 We then (pointwise) define the vector spaces $\SD_\perp^i:=\langle {\mathfrak{Br}}_{\perp, t}^{i}\rangle$ and $\SD^i_t:=\langle\mathfrak{Br}^{i}_t\rangle$, where 

\begin{equation}\label{SDsum}
\SD_i=\SD_\perp^i + \SD^i_t.
\end{equation}

\begin{remark}
The notation chosen for denoting the planes $\SD_\perp^i $ and $ \SD^i_t$ will be justified in the context of adapted $r-1$-jets by Remark \ref{Terminology}. We may write $\SD_\perp^i(F)$ and $ \SD^i_t(F)$ whenever we want to make it explicit that these objects are associated to a specific jet $F=(F_1,\cdots, F_k)\in J^{r-1}( \overline{\bigoplus_k T\R^n})$.
\end{remark}

\begin{remark} Note that the sum in equation (\ref{SDsum}) is not necessarily a direct sum in general. Lemma \ref{Dsplits} will show that it actually is under additional hypothesis.
\end{remark}


Denote by $\mathfrak{h}^i_m, \mathfrak{p}^i_m$ the length-$i$ multi-indices of the form $\mathfrak{h}^i_m:= (t,\cdots, t, m)\in\mathcal{I}^t_i$ and $\mathfrak{p}^i_m=(t,\cdots, m, t)\in\mathcal{I}^t_i$, respectively. 

\begin{remark}\label{BracketsOpposite} Note that, by Lemma \ref{antisymmetry}, the only length-$i$ brackets in $\mathfrak{Br}_t^i(F)$ which are potentially non-zero are the ones of the form $A_{\mathfrak{h}^i_m}(F)$ and $A_{\mathfrak{p}^i_m}(F)$, $m\neq t$. Moreover, Lemma \ref{antisymmetry} also implies that the following equality holds: $A_{\mathfrak{h}^i_m}(F)=- A_{\mathfrak{p}^i_m}(F)$. Thus, brackets $A_{\mathfrak{h}^i_m}(F)$, $m\geq 2$, are generators of $\SD_t^i(F)$.
\end{remark}

\begin{lemma}\label{Dsplits}
If $F\in\mathcal{S}^{\step -r}$ and $i<r$, then
\begin{itemize}
\item[i)] $\SD_i(F)=\SD_\perp^i(F) \oplus \SD^i_t(F)$ and, moreover,
\item[ii)] $\dim(\SD^i_t(F))=k-1$.
\end{itemize}
\end{lemma}

\begin{proof}
 First note that, for $i<r$, the elements $\SD_i(F)/\SD_{i-1}(F)$ have the same dimensions as the subspaces $\mathfrak{Lie}_n^i$ in the free Lie algebra $\mathfrak{Lie}_n$ (by Definition \ref{formalBG} together with Lemma \ref{gvMax}). Take a Hall basis $H$ associated to the ordered set $\lbrace F_1,\cdots, F_k\rbrace$. The subset $\mathcal{V}_i$ (Definition \ref{hall}) provides thus a basis of  $\SD_i(F)/\SD_{i-1}(F)$. By Corollary \ref{HallBasisElements}, the $k-1$ elements $A_{\mathfrak{h}^i_m}(F), (m=2,\cdots, k)$ cannot be expressed as combinations of elements from  $\mathfrak{Br}^{i}_\perp$, yielding $i)$. Additionally, they conform a basis of $\mathfrak{Br}^{i}_t$ since they are linearly independent (by Corollary \ref{HallBasisElements}) and they generate $\mathfrak{Br}^{i}_t$ (by Remark \ref{BracketsOpposite}). This yields $ii)$. \end{proof}

The following result is just a rephrasing of the notion of (formal) growth vector from Definition \ref{FormalDi} in terms of Equation (\ref{SDsum}). We will phrase it as a lemma.

\begin{lemma}
The condition for $\SD$ having growth vector $\nu_\SD=(\mathfrak{n}_i)_{i=1}^r$ is equivalent to:
\begin{equation}\label{conditionBG2}
\text{ for each }i\geq 1,\text{ }\dim\left( {\SD}_{\perp}^{i} + \SD_t^{i}\right) = \mathfrak{n}_i. 
\end{equation}

 or, equivalently, 
\begin{equation}\label{conditionBG}
\text{ for each }i\geq 1,\text{ }\dim\left(\SD_t^{i}/{\SD}_{\perp}^{i}\right)=\mathfrak{n}_i-\dim({\SD}^i_{\perp}). 
\end{equation} 

\end{lemma}

\begin{lemma}\label{CasesRanks}
Let $F\in\mathcal{S}^{\step -r}$ and $m_i:=\dim\left(\SD^i_\perp(F)\right)$. The following statements hold:
\begin{itemize}
\item[i)] if $i=r$, then $\mathfrak{n}_i=n$,
\item[ii)] if $i<r$, then $m_i+k-1=\mathfrak{n}_i<n$.
\end{itemize}
\end{lemma}

\begin{proof}
The first statement $i)$ just follows from $F$ being formally bracket generating of $\step$-$r$.  The first equality in $ii)$ follows from Equation (\ref{conditionBG2}) together with Lemma \ref{Dsplits}, whereas the inequality $\mathfrak{n}_i<n$ follows from the fact that we are not in the last step; i.e. $\SD_i\neq \SD_r=T\R^n$.
\end{proof}

\subsubsection{Principal subspaces and adapted frames}

The following two Lemmas state that the only brackets in $\mathfrak{Br}^{\mathfrak{m}}$ that depend on $j^{\mathfrak{m}-1}_t(F)$ are the ones in $\mathfrak{Br}^{\mathfrak{m}}_t$.

\begin{lemma}\label{Lemma:Dperp} Take a $\partial_t$-adapted jet $F=(F_i)_{i=1}^k\in J^{r-1}( \overline{\bigoplus_k T\R^n})$. Consider a length $\mathfrak{m}\geq 2$ formal bracket of the form $A_J(F)$, $J\in\tilde{\mathfrak{I}}_{\mathfrak{m}}$ (recall Eq. (\ref{Br2types})). Then $A_J(F)$ does not depend on $j^{\mathfrak{m}-1}_t(F)$; i.e. the bracket expression factors, for adapted $r-1$-jets, through some map $h$ as follows:
\begin{equation*}\label{BracketR2}
A_J(F)={h}\circ  j^{\bot(\partial_t,\mathfrak{m-1})}(F)
\end{equation*}

\end{lemma}

\begin{proof}
Let $J=(\ell_{\mathfrak{m}}, \cdots, \ell_1)$. Since $J\in\tilde{\mathfrak{I}}_{\mathfrak{m}}$, then the expression $A_J(F)=[F_{\ell_\mathfrak{m}},[ \cdots [F_{\ell_3},[F_{\ell_2}, F_{\ell_1}]],\cdots]$ belongs to ${\mathfrak{Br}}_{\perp, t}^{\mathfrak{m}}$ and it involves at least two components $F_{\ell_a}, F_{\ell_b}$ different from $F_1$. Now, since $F$ is $\partial_t$-adapted, then $j^0(F_{\ell_a}), j^0(F_{\ell_b})\perp\partial_t$ and the claim follows from Lemma \ref{PerpendicularEffect1}.
\end{proof}

By Remark \ref{BracketsOpposite}, brackets of the form $A_I(F)$ with $I=(t,\cdots, t, j)\in \mathfrak{I}_{\mathfrak{i}}^t$ constitute a set of generators of the whole $\SD_t^i(F)$. We thus focus our study on those jets in the following Lemma.

\begin{lemma}\label{Lemma:Dt} Let $I=(t,\cdots, t, j)\in \mathfrak{I}_{\mathfrak{m}}^t$ where $\mathfrak{m}\geq 2$, $j\neq t$. Then the bracket $A_I(\cdot)$ admits the following decomposition for $\partial_t$-adapted jets  $F=(F_i)_{i=1}^k\in J^{r-1}( \overline{\bigoplus_k T\R^n})$:

\begin{equation}\label{BracketT}
A_I(F)= P^{\mathfrak{m}-1}_t(F_j)+\nu_j(F_j)
\end{equation} 

(recall Definition \ref{FormalDT}) where $\nu_j(F_j)$ is an expression not depending on $j^\mathfrak{m-1}_t(F)$ or, more precisely: there is a map $h_j$ such that for $\partial_t$-adapted jets $F$, $\nu_j(F_j)={h}_j\left( j^{\bot(\partial_t,\mathfrak{m-1})}(F_j)\right)$.
\end{lemma}

\begin{proof}

Since $F$ is $\partial_t$-adapted, recall from  Eq. (\ref{AdaptedInCoordinates}) and (\ref{AdaptedInCoordinates2}) that we can write in jet-coordinates (introduced in Subsection \ref{jetcoordinates}):
\begin{align}\label{EqJetF1}
&F_t=(1, u_t^2, \cdots, u_t^n, \cdots, u_{t, J}^\ell,\cdots),  \quad 1\leq \ell\leq n,\quad J\in\NS_{r-1},\\ \label{EqJetF2}
&F_j=(0, u_j^2,\cdots, u_j^n, \cdots, u_{j, J}^\ell,\cdots),   \quad 1\leq \ell\leq n, \quad J\in\NS_{r-1},
\end{align}
where $F_t=F_1$ (recall Remark \ref{RemarkPartialt}). We expand the bracket $A_I(F)$ by writing it in simplified form (Definition \ref{SimplifiedForm}) following Lemma \ref{LemmaSymbolBrackets} and we isolate the terms involving order-$\mathfrak{m}-1$ derivatives with respect to $\partial_t$ from the rest of the expression as follows:

\begin{equation}\label{desarrollo}
A_I(F)=[F_t,[F_t,\cdots, [F_t, F_j]\cdots]]=\sum_{i=1}^n \left({u_t^t}\right)^{\mathfrak{m}-1}\cdot\underbrace{D_t\circ\cdots\circ D_t}_{\mathfrak{m}-1} (u_j^i) \cdot \partial_i + \sum_{i=1}^n \nu^i_j\cdot \partial_i.
\end{equation}

Note that each term $ \left({u_t^t}\right)^{\mathfrak{m}-1}\cdot D_{(t,\cdots, t)}(u_j^i)$  on the right-hand side comes from the contribution of the $\mathfrak{m}-1$ $u_t^1$-entries of each $F_t$ together with the entry $u_j^i$ from $F_j$. Additionally, note from Eq. (\ref{EqJetF1}) that $u_t^1=1$.

Since $F_j$ does not have $u_j^t$-component  ($u_j^t=0$ by Eq. (\ref{EqJetF2})), there are no more order-$\mathfrak{m}-1$ derivations with respect to $\partial_t$ in the expression. So, $\sum_{i=1}^n \nu^i_j$ is just a polynomial expression not containing a single term involving order-$\mathfrak{m}-1$ derivatives $D_t\circ\cdots\circ D_t$.

By noting that $\sum_{i=1}^n D_t\circ\cdots \circ D_t (u_j^i)\cdot\partial_i= \sum_{i=1}^n u_{i, (t,\cdots, t)}^j$ and by denoting $\nu_j:=\sum_{i=1}^n \nu_j^i$, we conclude that $A_I(F)=P^{\mathfrak{m}-1}_t(F_j) +\nu_j$, where $\nu_j$ does not depend on $j^\mathfrak{m-1}_t(F)$ or, equivalently, is of the form $\nu_j={h}_j\left( j^{\bot(\partial_t,\mathfrak{m-1})}(F_j)\right)$ for some map $h_j$.
\end{proof}

\begin{remark}\label{Terminology}
Lemma \ref{Lemma:Dperp} implies that $\SD^i_\perp(F)$ only depends on  $j^{\bot(\partial_t,{i-1})}(F_j)$ whereas, by Lemma \ref{Lemma:Dt} together with Remark \ref{BracketsOpposite}, $\SD^i_t$ can be expressed as:
\begin{equation}
\SD^i_t(F)=\langle P^{i-1}_t(F_2) + \nu_2,\cdots, P^{i-1}_t(F_k)+\nu_k \rangle, \quad \text{ where }\nu_j={h}_j\left( j^{\bot(\partial_t,{i-1})}(F_j)\right) \text{ for some map }h_j.
\end{equation}

\end{remark}
Recall  (Lemma \ref{SplittingDt} and Lemma \ref{FormalDT}) that $P_t^i(\cdot)$ just encodes the formal order-$i$ pure derivative in the $\partial_t$-direction. We can thus give the following equivalent description of a principal subspace  $\Pr^{i}_{\partial_t, {F}}$  in terms of $\SD^{i+1}_\perp$ and $\SD^{i+1}_t$:

\begin{equation}\label{PpalSubspace}
  \Pr^{i}_{\partial_t, {F}} = \left\{ G\in  J^{i}\left( \overline{\bigoplus_k T\R^n}\right)\ \middle\vert \begin{array}{l}
    j^{\bot(\tau, i)}(G)=j^{\bot(\tau, i)}(F), \\
    \SD_{i+1}(G)=\overbrace{\color{black}\SD^{i+1}_\perp({F})}^{\color{black}\SD^{i+1}_{\perp}(G)}+\color{black}\overbrace{\color{black}\langle \color{blue} w_1\color{black}+\nu_2\color{blue},\cdots, w_{k-1}\color{black}+\nu_k\color{blue}\color{black}\rangle,}^{\color{black}\SD^{i+1}_t(G)} \color{black}\quad \color{blue} w_i\in\R^n \color{black}
  \end{array}\right\},
\end{equation}

where $\nu_j={h}_j\left( j^{\bot(\partial_t,{i-1})}(F_j)\right) \text{ for some map }h_j$. Blue has been used to denote the ``free'' variables that parametrize the space whereas ``black'' is reserved for the fixed variables determined by $F$.

\begin{remark}
The second condition in the description of the principal subspace $\Pr^{i}_{\partial_t, {F}}$ (Equation \ref{PpalSubspace}) is redundant, since the first condition $j^{\bot(\tau, i)}(G)=j^{\bot(\tau, i)}(F)$ already characterises $\Pr^{i}_{\partial_t, {F}}$. Nonetheless, we include it since it is descriptive and  will be useful in order to describe $  \mathcal{S}^{i-1}_{\partial_t, \tilde{F}}$ (see Equation (\ref{Sr})).
\end{remark}

\subsection{Ampleness of $\mathcal{S}^{\step -r}$}\label{AmplenessS}

This subsection is devoted to the proof of ampleness of $\mathcal{S}^{\step -r}$ for distributions of $\rank>2$ on manifolds of dimension $n\geq 3$. By the Fibered Ampleness Criterion (Proposition \ref{AmplenessCriterion}) this will imply Theorem \ref{AmplenessR} and, ultimately, Theorem \ref{mainthm}.

We will first state some lemmas about ampleness of subspaces of matrices that will be useful later in the proof of the Theorem. We isolate them here for clarity and because they are interesting on their own.

Let us fix some notation. Consider the space of $(\ell\times q)$-matrices with first $k$ columns fixed for some choice of $k<q$ vectors in $\R^q$ and denote it by $ \mathcal{A}_{\ell\times q} $. Denote by $ \overline{\mathcal{A}_{\ell\times q} }\subset \mathcal{A}_{\ell\times q} $ the subset of maximal rank matrices in $\mathcal{A}_{\ell\times q}$.  We will study ampleness of such subspace depending on the values of $k, \ell$ and $q$.

Let us begin studying the case $\ell=q$:

\begin{lemma}\label{MatricesGL}
If $q-k\geq 2$, then the space of maximal rank matrices with first $k$ columns fixed $ \overline{\mathcal{A}_{q\times q} }$ is ample within the space of all matrices with identical $k$ columns fixed $\mathcal{A}_{q\times q}$ for any choice of fixed columns.
\end{lemma}

\begin{proof}

If the first fixed $k$ columns are not $k$ linearly independent vectors in $\R^q$ then the space $\overline{\mathcal{A}_{q\times q} }$ is empty and trivially ample.

Otherwise, consider the projection map $\tilde{\pi}:\mathcal{A}_{q\times q} \to \mathcal{M}_{q-k\times q-k}$ that projects any matrix $A_{q\times q}$ to its submatrix  $M_{q-k\times q-k}$ resulting from removing its first $k$ columns and rows. By regarding the columns of a matrix $A$ as vectors in $\R^q$, this map can be thought as a quotient by the subspace spanned by the first $k$ column vectors. Ampleness now follows by ampleness of $GL(q-k)$ (Lemma \ref{LemmaGL}).
\end{proof}

We now study separately the case $\ell\neq q$. We will also assume in the statement that the first $k$ columns are linearly independent since, otherwise, ampleness follows trivially.

\begin{lemma}\label{MatricesThin}
Let $k<\ell < q$, then the space $ \overline{\mathcal{A}_{\ell\times q} }$ of maximal rank matrices with first $k$ columns linearly independent and fixed, is the complement of a thin singularity within the space of matrices with identical $k$ columns fixed $\mathcal{A}_{q\times q}$ and is, thus, ample.
\end{lemma}
\begin{proof}
The result follows from the fact that the space of non-maximal rank $(\ell\times q)$-matrices has codimension greater than $1$ (as a stratified subset) within the space of $(\ell\times q)$-matrices (see, for example, \cite[2.2.1]{EM}) .
\end{proof}

Note that if the first fixed $k$ columns constitute a maximal rank submatrix themselves, then ampleness follows trivially. This is encoded by the following  Lemma whose proof is trivial:

\begin{lemma}\label{MatricesTrivial}
Let $k= \ell < q$. The space $ \overline{\mathcal{A}_{\ell\times q} }$ of maximal rank matrices with first $k$ columns linearly independent and fixed coincides with the whole space of matrices with identical $k$ columns fixed $\mathcal{A}_{q\times q}$ and is, thus, trivially ample.
\end{lemma}

The case $\ell=q$ and $k=q-1$ is left for the last Section of the paper (see Lemma \ref{MatricesHyperplanes}).

Let us finally prove ampleness of $\mathcal{S}^{\step -r}\subset J^{r-1}\left( \overline{\bigoplus_k TM}\right)$ for $k>2$.

\begin{theorem}\label{SAmple}
 $\mathcal{S}^{\step -r}\subset J^{r-1}\left( \overline{\bigoplus_k TM}\right)$ is ample if $k>2$.
\end{theorem}
\begin{proof}
Since convex integration works locally and is implemented chart by chart, we assume $M=\R^n$. Consider a point $p\in\R^n$ and take $F=(F_1,\cdots, F_k)\in J^{r-1}\left( \overline{\bigoplus_k T\R^n}\right)$ over $p$. Take a principal codirection $\eta\in T_p^*\R^n$ which, by the duality given by the Euclidean metric, defines a principal direction $\partial_t$ (recall Remark \ref{IntegrableHyperplanes}). We will check ampleness with respect to $\partial_t$.

 By the Reduction Step 2, in order to check ampleness, we can assume that $\partial_t$ is non-normal to $\SD=\langle j^0(F_1),\cdots, j^0(F_k)\rangle$.
Choose another $\tilde{F}=(\tilde{F}_1,\cdots, \tilde{F}_k)\in J^{r-1}\left( \overline{\bigoplus_k TM}\right)$ such that $\pi^r(\tilde{F})=\pi^r(F)$ and so that it is an adapted $r-1$-jet with respect to the direction $\partial_t$ (Lemma \ref{ExistenceAdapted}); i.e.:

\begin{itemize}
\item[i)] $\pi^{r-1}(\tilde{F})=\pi^{r-1}(F)$; i.e. they define the same element in $J^{r-1}\left(\Gr_k(TM)\right)$ and
\item[ii)] $j^0(\tilde{F}_1)$ is the orthogonal projection of $\partial_t$ to $\SD$ with respect to the Euclidean metric of $\R^n$ .
\end{itemize}

For the ease of notation, we still denote by $\mathcal{S}^{\step -r}$  the projected relation $\pi^{r-1}_{i-1}\left(\mathcal{S}^{\step -r}\right)$. Recall that, by Lemma \ref{CondicionCriterio}, $\mathcal{S}^{\step -r}_{\partial_t, F}\subset\Pr_{\partial_t, F}^{i-1}$ is ample if and only $\mathcal{S}^{\step -r}_{\partial_t, \tilde{F}}\subset\Pr_{\partial_t, \tilde{F}}^{i-1}$ is ample. So, we will check ampleness of the latter.

Consider the subspaces $\SD^i_\perp(\tilde{F})$ and $\SD_t^i(\tilde{F})$ associated to $\tilde{F}$. By Equation (\ref{PpalSubspace}) we can express the principal subspace as:

\begin{equation}\label{PpalSubspaceA}
  \Pr^{i-1}_{\partial_t, {\tilde{F}}} = \left\{ G\in  J^{i-1}\left( \overline{\bigoplus_k T\R^n}\right)\ \middle\vert \begin{array}{l}
    j^{\bot(\partial_t, i-1)}(G)=j^{\bot(\partial_t, i-1)}(\tilde{F}), \\
    \SD_{i}(G)=\overbrace{\color{black}\SD^{i}_\perp(\tilde{F})}^{\color{black}\SD^{i}_{\perp}(G)}+\color{black}\overbrace{\color{black}\langle \color{blue} w_1\color{black}+\nu_2\color{blue},\cdots, w_{k-1}\color{black}+\nu_k\color{blue}\color{black}\rangle,}^{\color{black}\SD^{i}_t(G)} \color{black}\quad \color{blue} w_i\in\R^n \color{black}
  \end{array}\right\}
\end{equation}

Combining (\ref{PpalSubspaceA}) with (\ref{conditionBG}) the intersection $\mathcal{S}^{i-1}_{\partial_t, \tilde{F}}:=\Pr^{i-1}_{\partial_t, {\tilde{F}}}\cap \mathcal{S}^{\step -r}$ of the principal subspace with $\mathcal{S}^{\step -r}$ reads as:

\begin{equation}\label{Sr}
  \mathcal{S}^{i-1}_{\partial_t, \tilde{F}} = \left\{ G\in  J^{i-1}\left( \overline{\bigoplus_k T\R^n}\right)\ \middle\vert \begin{array}{l}
j^{\bot(\partial_t, i-1)}(G)=j^{\bot(\partial_t, i-1)}(\tilde{F}), \\
\dim \big( \color{black}\langle \overbrace{\color{blue}w_1\color{black}+\nu_2\color{blue},\cdots, w_{k-1}\color{black}+\nu_{k}}^{\color{blue} w_i \text{ are free, } \color{black} \nu_i \text{ are frozen.} }\rangle\color{black}/\overbrace{\color{black}\SD^{i}_\perp(\tilde{F})}^{\color{black}\text{frozen data}}\big)\color{black}=\overbrace{\color{black}\mathfrak{n}_i-\dim(\SD^{i}_\perp(\tilde{F}))}^{\text{frozen data}}, \quad \color{blue}  w_i\in\R^n \color{black}
  \end{array}\right\}
\end{equation}

Note that all data are fixed except for the $k-1$ vectors $w_i \in\R^n$ corresponding to the pure formal partial derivatives of order $i-1$ with respect to the principal direction $\partial_t$. Denote $m_i:=\dim(\SD^i_\perp(\tilde{F}))$. 

Note that in order to check ampleness of  $\mathcal{S}^{i-1}_{\partial_t, \tilde{F}}$ within $ \Pr^{i-1}_{\partial_t, {F}}$, the constants $\nu_2,\cdots, \nu_k$ are irrelevant since they are fixed and can thus be subsummed within the free variables \color{blue}$w_1,\cdots, w_{k-1}$\color{black}. We can, thus, assume without loss of generality that $\nu_2=\cdots=\nu_k=0$.

Take a generating set $\lbrace v_1,\cdots, v_{m_i}\rbrace$ of $\SD^i_\perp(\tilde{F})$; i.e. so that $\langle v_1,\cdots, v_{m_i}\rangle=\SD^i_\perp(\tilde{F})$. Define $\mathcal{B}_{n\times(m_i+k-1)}\subset \mathcal{M}_{n\times(m_i+k-1)}$ as the subset of matrices $B\in\mathcal{M}_{n\times(m_i+k-1)}$ whose first $m_i$ columns are formed by the vectors $v_1,\cdots, v_{m_i}$.
Additionally, denote by $\overline{\mathcal{B}}_{n\times(m_i+k-1)}\subset\mathcal{B}_{n\times(m_i+k-1)}$ the subset of matrices in $\mathcal{B}_{n\times(m_i+k-1)}$ of $\rank(B)=\mathfrak{n}_i$, where $\mathfrak{n}_i$ denotes the $i$-th entry of a maximal growth vector (Remark \ref{notationni}).

\begin{figure}[h]
	\includegraphics[scale=1]{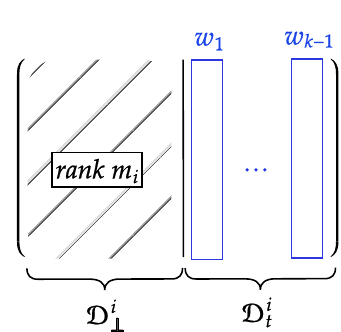}
	\centering
	\caption{The subspace $\mathcal{B}_{n\times(m_i+k-1)}$ is formed by those matrices with first $m_i$ columns fixed (striped area) and generating the $m_i$-dimensional subspace $\SD^i_\perp(\tilde{F})$. The other $k-1$ columns (in blue) are not fixed. Checking ampleness of $\mathcal{S}^{i-1}_{\partial_t, \tilde{F}}\subset\Pr_{\partial_t, \tilde{F}}^{i-1}$ tantamounts to checking ampleness of the space of $\mathfrak{n}_i$-rank matrices $\overline{\mathcal{B}}_{n\times(m_i+k-1)}\subset \mathcal{B}_{n\times(m_i+k-1)}$.}\label{matrixAmple}
\end{figure}

Note that showing ampleness of $\mathcal{S}^{i-1}_{\partial_t, \tilde{F}}\subset\Pr_{\partial_t, \tilde{F}}^{i-1}$ tantamounts to showing ampleness of the space $\overline{\mathcal{B}}_{n\times(m_i+k-1)}\subset\mathcal{B}_{n\times(m_i+k-1)}$. We distinguish several cases in order to conclude:

\begin{itemize}

\item[i)] If $i\neq r$, then $m_i+k-1=\mathfrak{n}_i<n$ (Lemma \ref{CasesRanks}) and ampleness of $\overline{\mathcal{B}}_{n\times(m_i+k-1)}\subset\mathcal{B}_{n\times(m_i+k-1)}$ tantamounts to ampleness of the subset of $\rank=m_i+k-1<n$ matrices (maximal rank) with $m_i$ columns fixed inside the space of all matrices with the same columns fixed. This subset has a {thin} complement by Lemma \ref{MatricesThin} and, thus, ampleness follows.

\item[ii)] If $i=r$, then $\mathfrak{n}_i=n$ (Lemma \ref{CasesRanks}) and ampleness of $\overline{\mathcal{B}}_{n\times(m_r+k-1)}\subset\mathcal{B}_{n\times(m_r+k-1)}$ tantamounts to ampleness of the subset of $\rank=n$ matrices (maximal rank) with $m_r$ columns fixed inside the space of all matrices with the same columns fixed. Note that the case $n>m_r+k-1$ cannot take place since this would contradict the fact that $F$ was a formal solution in the first place. We have thus two possible cases:

\begin{itemize}

\item[ii.a)] If $n<m_r+k-1$ then there are two possibilities. If $m_r=n$, then ampleness holds trivially by Lemma \ref{MatricesTrivial}. Otherwise,  the space of maximal rank matrices with $m_r$ columns fixed has a thin complement in $\mathcal{B}_{n\times(m_r+k-1)}$ by Lemma \ref{MatricesThin} and ampleness follows as well.

\item[ii.b)] If  $n=m_r+k-1$, ampleness follows since the space of maximal rank matrices with $m_r$ columns fixed in $\mathcal{B}_{n\times(m_r+k-1)}$ is ample if $k-1\geq 2$ by Lemma \ref{MatricesGL} (although it does not have a thin complement, recall Remark \ref{NonThinnesOfGL}).

\end{itemize}

\end{itemize}
This completes the proof.\end{proof}

The following result, which we isolate here since it has interest on its own, follows from the proof of Theorem \ref{SAmple} (precisely, from case $i)$ in the case distinction).
\begin{lemma}
Fix $i<r$. Let $p\in\R^n$ and let $\partial_t\in T_p\R^n$ be a non-normal direction. Then for a $\partial_t$-adapted jet $\tilde{F}$ over $p$, $\mathcal{S}^{i-1}_{\partial_t, \tilde{F}}$ has a thin complement within $\Pr_{\partial_t, \tilde{F}}^{i-1}$.
\end{lemma}

\begin{remark}[Associated singularities are not thin in the last step of the free case]
Note that, for $i=r$ in the free case (Definition \ref{FreeType}), the equality $m_r+k-1=n$ always holds (Lemma \ref{Dsplits}) and thus case $ii.a)$ in the case distinction above can never take place. This means that $\mathcal{S}^{r-1}_{\partial_t, \tilde{F}}\subset\Pr_{\partial_t, \tilde{F}}^{r-1}$ is either trivially ample or its ampleness follows by ampleness of $GL(m_r+k-1)$. Therefore, either ampleness of $\mathcal{S}^{r-1}_{\partial_t, \tilde{F}}\subset\Pr_{\partial_t, \tilde{F}}^{r-1}$ holds trivially or it holds without yielding a thin singularity (Remark \ref{NonThinnesOfGL}). 
\end{remark}

\section{The rank-$2$ case}

In this last Section we treat separately the $\rank$-2 case. We will show that the maximal growth condition for rank-$2$ distributions does not give raise to ample differential relations. By the arguments carried out until this point (see First Reduction, Subsection \ref{Red1}), ampleness of $\mathcal{R}^{\step -r}$ is equivalent to ampleness of $\mathcal{S}^{\step -r}$. We will show that $\mathcal{S}^{\step -r}$ is not ample for rank-$2$ distributions, thus showing that neither is $\mathcal{R}^{\step -r}$.

\begin{lemma}\label{MatricesHyperplanes}
The space $\overline{\mathcal{A}_{q\times q} }$ of maximal rank matrices with the first $q-1$ columns fixed (and linearly independent) is the complement of a hyperplane within the space $\mathcal{A}_{q\times q}$ of all matrices with identical fixed columns. Therefore, $\mathcal{A}_{q\times q}\subset \overline{\mathcal{A}_{q\times q} }$  is not ample.
\end{lemma}

\begin{proof}

The space $ \overline{\mathcal{A}_{q\times q} }$ has two connected components, the one corresponding to matrices with positive determinant and the one with negative determinant. The complement of such space in $\mathcal{A}_{q\times q}$ corresponds to matrices with zero determinant. These matrices are determined by a linear equation (given by $det(A)=0$, which is linear in the entries of the non-fixed column) and thus represent a hyperplane within $\mathcal{A}_{q\times q}$. Non-ampleness follows (recall Example \ref{hyperplane}).
\end{proof}

\begin{theorem}
Let $M$ be a differentiable manifold of dimension $n\geq 3$ and consider the differential relation $\mathcal{S}^{\step -r}\subset J^{r-1}( \overline{\bigoplus_2 TM})$. Along any principal subspace, $\Pr^{r-1}_{\partial_t, F}\cap\mathcal{S}^{\step -r}$ is either not ample or trivially ample.
\end{theorem}
\begin{proof}
Take $F=(F_1,\cdots, F_k)\in J^{r-1}( \overline{\bigoplus_2 TM})$ over a point $p\in M$ and a principal direction $\partial_t\in T_pM$. If $\partial_t\perp\SD(F)$, then ampleness follows trivially (Corollary \ref{TrivialAmpleness}). 
In other case, we can reproduce verbatim the same argument from the proof of Theorem \ref{SAmple}. In that case, denoting by $m_r=\rank\left(\SD_\perp^r\right)$, we have that ampleness of  $\Pr^{r-1}_{\partial_t, F}\cap\mathcal{S}^{\step -r}$ is equivalent to ampleness of the space of $\rank$-$n$ matrices with $m_r$ columns fixed (i.e. all but one) $\overline{\mathcal{A}}_{n\times(m_r+1)}$ within the space of all matrices with the same fixed columns $\mathcal{A}_{n\times(m_r+1)}$. 

Considering the same case distinction as in the proof of Theorem \ref{SAmple}, and following the same notation for its subcases: ii.a) and ii.b), we encounter two possibilities:

\begin{itemize}
\item[ii.a)] If $n<m_r+1$ this means that $m_r=n$ since $m_r$ cannot be greater than $n$ for dimensional reasons. Therefore, ampleness follows trivially by Lemma \ref{MatricesTrivial}.

\item[ii.b)] If  $n=m_r+1$, then non-ampleness follows since the space $\overline{\mathcal{A}}_{n\times n}$ of maximal rank matrices with $m_r=n-1$ columns fixed within the space of all matrices with the same columns fixed $\mathcal{A}_{n\times n}$ is non-ample (with its singularity being a hyperplane) by Lemma \ref{MatricesHyperplanes}.
\end{itemize}

This completes the proof.\end{proof}

The fibered-criterion of ampleness readily implies Theorem \ref{NonAmplenessRank2}.

\begin{remark} Non-ampleness of the Engel and the $(2,3,5)$-conditions follow as particular cases.
\end{remark}

\end{document}